\documentclass[11pt]{article}
\usepackage{amsmath}
\usepackage{graphicx}

\newtheorem{theorem}{Theorem}

\newtheorem{prop}[theorem]{Proposition}
\newenvironment{proof}{{\em Proof.}}{\hfill\rule{8pt}{8pt}\vskip10pt}

\newtheorem{ex}{Example}

\newtheorem{algorithm}{Algorithm}

\newcommand{\bm}[1]{\boldsymbol{#1}}

\newcommand{\De}{\partial}
\def\D{\,{\rm d}}
\def\R#1{$(\ref{#1})$}
\def\E{{\rm e}}
\newcommand{\norm}[1]{\| #1 \|}

\begin{document}

\title{The Euler equation of quasi-geostrophic fluids and volume preserving numerical methods}
\author{Antonella Zanna}
\date{\today}
\maketitle

\begin{abstract}
We consider the Euler equation of quasi-geostrophic fluids which is widely used in weather forecast. 
Our goal is to study explicit volume-preserving numerical methods for very long simulations on an energy and enstrophy preserving discretization. To this purpose, we compute the average fields and estimate statistical parameters. It is observed that the statistical parameters depend on the integrals of the discretization and that the computed parameters are affected by the error in those, even if the volume of the phase space is correctly preserved. We conclude that monitoring the error in the integrals by the explicit volume preserving method can be used as an indication of how good the estimated parameters are.
\end{abstract}

\section{Introduction}
This paper deals with the problem of computing long-time trajectories for the Euler equations of quasi-geostrophic fluids by means of volume preserving numerical methods. Such equations are widely used in applications like weather and climate predictions on Earth as well as for other gas planets, such as Jupiter. Due to the highly inhomogeneities in the atmosphere or ocean, there is a wide range of unresolved scales of motion and the system displays a seemingly chaotic behaviour. As a consequence, it is impossible to simulate particular solutions with a prescribed accuracy. The application of ideas from statistical mechanics to these geostrophic flows, on the other hand, has lead to promising predictions for the ocean, atmosphere and giant planets, in agreement with contemporary observations \cite{majda2006non}.  

The statistics are obtained by performing very long numerical simulations for some initial condition. Because of the seemingly chaotic behaviour and the large number of unresolved scales, the numerical approximation produced will be far off the trajectory one was set to simulate. However, as long as the iteration map is ergodic, time averages are equal to space averages almost everywhere.\footnote{Under some continuity assumptions on the map, the measure and the characteristic function, the "almost everywhere" convergence can be replaced by "everywhere" convergence.} If the iteration map is ergodic, the long time trajectory will visit a large number of points in the phase space with the same probability, and the time averages will converge to the spatial averages. Ergodicity is equivalent to the preservation of an invariant measure (preservation of a volume form).
Thus it makes sense to study the effect of volume preserving integrators on the long-time dynamic of such systems.
Our starting point will be  Arakawa's energy and enstrophy preserving semi-discretization \cite{arakawa1966computational}, which has been shown to preserve volume in phase space \cite{MR2442396}. 

We will investigate an explicit volume preserving numerical integrator based on splitting methods. The method does not preserve exactly the underlying integrals of the semi-discretization, however it is  self-adjoint and therefore expected to preserve some nearby integrals by virtue of backward error analysis. In order to construct an explicit splitting method, we will show that the $i$th component of the vector field does not depend on the $i$th variable -- thus the differential equation is \emph{diagonal free} in addition to be divergence free (the latter was shown in \cite{MR2442396}). This will allow us to split the vector field in \emph{canonical shears}, which are volume preserving, see \cite{mclachlan04egi,mclachlan09evp}. To retain the right complexity, the splitting must be performed in such a way that computing one step of the explicit method does not cost more than  evaluating the vector field  ($O(N^4)$ operations, where $N^2$ is the number of variables), which is a tricky task due to the presence of the Laplacian operator, that mixes up all the variables. This will be discussed in Section~\ref{sec:constructionvp}. The simplest way to implement the canonical splitting is by solving in a sequential order for the variables. However, our preliminary numerical tests showed that this yields a poor method in terms of error propagation. It is known that chessboard-like ordering of the variables can yield better error propagation for some applications (see for instance \cite{frank97gif}). We will also propose a new ordering method (see \S\ref{sec:indexing}) that is obtained by minimizing a measure of the error of the splitting method: we identify the commuting vector fields, collect them together, and add, at each step, those that whose commutation error is minimal among the remaining one. This aggregation gives a significant improvement for the error accumulation of the numerical simulations.
We perform a number of long term simulations, for the different indexing strategies and different number of variables in \S\ref{sec:numerical}. 

The conclusions of this work can be summarized as follows: for a relatively small number of discretization points, the second order explicit volume preserving numerical method gives satisfactory results when implemented with our indexing strategy, and the errors on the conserved quantities stay small over a very long integration time, yielding statistically reliable estimated parameters. However, when the number of discretization points increase, there is a deterioration of the long-time conservative properties of the methods, even with our proposed indexing strategy. Thus, it might be appropriate to increase the order of the method. 
Although it is disappointing to observe this effect of the dimensionality on the long-time conservative properties of the proposed volume preserving scheme, there is also a constructive side, as the error on the conserved quantities can be used to monitor the error on the statistical quantities of interest and hence is a measure of their reliability.

\section{The quasi-geostrophic model}
We will consider one of the  simplest geophysical models, a barotropic two-dimensional flow with topography  on a periodic domain. The model is described by the differential equations
\begin{displaymath}
	q = \Delta \psi+h, \qquad \frac{\De q}{\De t} + \nabla \psi \times \nabla q = 0,
\end{displaymath}
where $q$ is the vorticity, $\psi$ is the stream function and  $\nabla u = [\De_x u, \De_y u]^T$ and $h$ is the topography.
All functions are defined on a a square $[0, 2\pi]\times [0,2\pi]$, and we assume periodic boundary conditions.
Due to the commuting properties of the differential operators and the differentiation rules, the equation can be written in the form
\begin{equation}
	\frac{\De q}{\De t} -\mathcal{J}_\mathcal{M}(q,\psi) = 0,
\end{equation}
where $\mathcal{J}_\mathcal{M}$ is any of the following equivalent formulations
\begin{eqnarray}
	\mathcal{J}_0(u,v) &=& \nabla u \times \nabla v \label{eq:J0}\\ 
	\mathcal{J}_\mathcal{E}(u,v) &=& \nabla \times (u \nabla v)\label{eq:JE}\\
	\mathcal{J}_\mathcal{Z}(u,v) &=& -\nabla \times(v \nabla u). \label{eq:JZ}
\end{eqnarray}
The system is Hamiltonian, with Poisson bracket
\begin{equation}
	\{F, G\} = \int q \mathcal{J}(\nabla F, \nabla G ) \D \mathbf{x}
	\label{eq:Pbracket}
\end{equation}
and Hamiltonian energy functional
\begin{displaymath}
	\mathcal{E}(q) = -\frac12 \int \psi (q-h) \D \mathbf{x}.
\end{displaymath}
The Poisson bracket is degenerate and the system possesses an infinite set of Casimir invariants, of the form
\begin{displaymath}
	\int f(q) \D \mathbf{x},
\end{displaymath}
where $f$ is any scalar function. It is customary to consider invariants of the form $\int q^n \D \mathbf{x}$, for $n=1,2, \ldots$ (vorticity moments). The most important are the the first moment, also referred as the \emph{total circulation}
\begin{equation}
	\mathcal{C}(q) = \int q \D \mathbf{x}
	\label{eq:total_circ}
\end{equation}
(or total vorticity), and the second moment of vorticity, also called \emph{enstrophy},
\begin{equation}
	\mathcal{Z}(q) = \frac12 \int q^2 \D \mathbf{x}. 
	\label{eq:enstrophy}
\end{equation}
The higher moments do not seem to have much relevance in applications, except for the \emph{third moment}
\begin{displaymath}
	\frac13 \int q^3 \D \mathbf{x},
\end{displaymath}
which has been seen to have some relevance in the statistical analysis of the discretized flow \cite{abramov2003statistically}.

\section{The conservative Arakawa decompositions and volume preservation}
Arakawa \cite{arakawa1966computational} considered semi-discretizations of the above equations, in which the differential operators were replaced by standard central divided difference operators:
\begin{displaymath}
	\De_x \approx D_x, \qquad \De_y \approx D_y. 
\end{displaymath}
In the discussion that follows, it is important that the operators are skew-symmetric, i.e.\ $D_x^T = -D_x, D_y^T=-D_y$. We consider a symmetric discrete approximation to the Laplacian,  $\Delta_N= \Delta_N^T$, and, for the time being, we assume it to be invertible.\footnote{When this is not the case, a symmetric pseudo-inverse $\Delta_N^{-1}$  will be used instead, see discussion later.}
We introduce a discretization of the domain. Thus, for an arbitrary function $u$, we set  $u_{i,j} \approx u(i \delta, j \delta)$, where $\delta= 2\pi/N$ is the spacing in the $x$ and  $y$ directions (isotropic discretization) and $N$ is the number of mesh points in each direction (for a total of $N^2$ grid points and as many variables). Further, by $\mathbf{u}$, we will denote the vector of the $u_{i,j}$ using the standard column-wise ordering.
Letting 
\begin{displaymath}
	\bm{\psi} = \Delta_N^{-1} (\mathbf{q-h}),
\end{displaymath}
we can express the equations in terms of the sole variable $\mathbf{q}$, so that the semi-discretized system takes the form
\begin{equation}
	\dot{\mathbf{q}} = J_\mathcal{M}(\mathbf{q}),
	\label{eq:JMflow}
\end{equation}
where $J_\mathcal{M}$ is one of the expressions below:
\begin{eqnarray}
	J_0(\mathbf{q}) &=& (D_x \mathbf{q}) * (D_y \Delta_N^{-1} (\mathbf{q-h})) -  (D_y \mathbf{q}) * (D_x \Delta_N^{-1} (\mathbf{q-h})), \label{eq:J0d} \\
	J_\mathcal{E}(\mathbf{q}) &=& D_x (\mathbf{q} * D_y \Delta_N^{-1} (\mathbf{q-h})) - D_y (\mathbf{q} * D_x \Delta_N^{-1} 
	(\mathbf{q-h})), \label{eq:JEd}\\
	J_\mathcal{Z}(\mathbf{q}) &=& D_y((D_x \mathbf{q}) * \Delta_N^{-1} (\mathbf{q-h})) - D_x((D_y \mathbf{q}) * \Delta_N^{-1} (\mathbf{q-h})), \label{eq:JZd}
\end{eqnarray}
$\mathcal{M}=0,\mathcal{E},\mathcal{Z}$, which are finite difference discretizations of the corresponding formulations \R{eq:J0}-\R{eq:JZ}.  The `$*$' operator denotes componentwise multiplication of vectors, $(\mathbf{u}* \mathbf{v})_i = u_i v_i$, so that the product
\begin{displaymath}
	\mathbf{u}^T (\mathbf{v}* \mathbf{w}) = \sum_i u_i v_i w_i = \mathbf{v}^T (\mathbf{u}* \mathbf{w}) = \mathbf{w}^T (\mathbf{u}* \mathbf{v})
\end{displaymath}
is fully symmetric.
Differently from the analytical case, the semi-discretizations \R{eq:J0d}-\R{eq:JZd} \emph{are not} equivalent, rather, they have very different properties. All of them preserve the discrete version of the total vorticity,
\begin{equation}
	C(\mathbf{q}) = \mathbf{q}^T \mathbf{1} \delta^2,
	\label{eq:C_N}
\end{equation}
otherwise, the flows have different conservation properties.
Arakawa showed that the choice \R{eq:JEd} preserves a discrete version of the energy,
\begin{equation}
	E(\mathbf{q}) = -\frac12 \bm{\psi}^T (\mathbf{q-h})\delta^2,
	\label{eq:E_N}
\end{equation}
(but not the enstrophy), while \R{eq:JZd} preserves the discrete enstrophy,
\begin{equation}
	Z(\mathbf{q}) = \frac12 \mathbf{q}^T \mathbf{q} \delta^2
	\label{eq:Z_N} 
\end{equation}
(but not the energy). $J_0$ preserves neither the discrete energy \R{eq:E_N} nor the enstrophy \R{eq:Z_N}.

In addition to \R{eq:J0d}-\R{eq:JZd}, Arakawa considered also a semi-discretized flow that preserves energy \emph{and} enstrophy (and also linear momentum),
\begin{equation}
	J_\mathcal{EZ} (\mathbf{q}) = \frac13 (J_{0} (\mathbf{q})+J_\mathcal{E} (\mathbf{q}) + J_\mathcal{Z} (\mathbf{q})).
	\label{eq:JEZd}
\end{equation}
The proofs of conservation can be found in the original work of Arakawa. A more modern approach, based on Nambu bracket formalism, can be found in \cite{MR2442396}.

The semi-discretizations \R{eq:J0d}-\R{eq:JZd} and \R{eq:JEZd} are also \emph{volume} preserving. A proof, using the skew-symmetry of the discretized differential operators $D_x, D_y$ and the symmetry of $\Delta_N^{-1}$, as well as the properties of the trace operators, can be found in \cite{MR2442396}.


\section{Volume preserving numerical methods}
The divergence-free differential equation
\begin{displaymath}
	\left\{ \begin{array}{l}
	\dot{\mathbf{q}} =\mathbf{f}(\mathbf{q}),\\
	\mathbf{q}(0) = \mathbf{q}_0,
	\end{array} \right.
	 \qquad \nabla^T \mathbf{f} = \sum_{i} \frac{\partial{f_i}}{\De q_i} = 0,
\end{displaymath}
preserves the volume, as it has the property that $\det \partial \mathbf{q}/\partial \mathbf{q}_0=1$.
Similarly, a numerical method for $\dot{\mathbf{q}} =\mathbf{f}(\mathbf{q})$, 
\begin{displaymath}
	\mathbf{q}_{n+1} = \bm{\Psi}_{\tau,\mathbf{f}} (\mathbf{q}_n), \qquad n=0,1,\ldots,
\end{displaymath}
is  \emph{volume preserving} provided that
\begin{equation}
	\det \frac{\De \bm \Psi_{\tau, \mathbf{f}}(\mathbf{q} )}{\De \mathbf{q}}=1,
	\label{eq:vp}
\end{equation}
or, equivalently, $\det [{\De \mathbf{q}_{n+1} }/{\De \mathbf{q}_{n}}] =1$ for all $n$ (whenever $\nabla^T \mathbf{f}=0$). Following \cite{mclachlan09evp}, a system obeying the condition
\begin{displaymath}
	\frac{\partial{f_i}}{\De q_i} = 0, \qquad i=1, 2, \ldots,
\end{displaymath}
is called \emph{diagonal-free} system, since the $i$-th component does not depend on the $i$-th variable. Diagonal-free systems are obviously divergence free, but the converse is not true, as there exists divergence-free system that are not diagonal free. 
It is easier to design volume preserving numerical methods for diagonal-free systems, while the general case with nonzero diagonal contributions is typically more complicated \cite{mclachlan09evp}.

Let us consider a splitting method, based on the splitting of the function $\mathbf{f}$ in $M$ ``sub vector fields'' 
\begin{displaymath}
	\mathbf{f} = \mathbf{F}^1 + \cdots + \mathbf{F}^M,
\end{displaymath}
each of them also divergence free. Assume that each vector field $\mathbf{F}^k$ is  integrated exactly by $\bm{\Psi}_\tau^k$, for $k=1, \ldots, M$. Then, the composition
\begin{displaymath}
	\bm{\Psi}_{\tau, \mathbf{f}} = \bm{\Psi}_\tau^1 \circ \cdots \circ  \bm{\Psi}_\tau^M
\end{displaymath}
is also volume preserving and yields an order one method. We will consider the \emph{symmetric} composition,
\begin{equation}
	\bm{\Psi}_{\tau,\mathbf{f}} = \bm{\Psi}_{\tau/2}^1 \circ \cdots \circ \bm{\Psi}_{\tau/2}^{M-1} \circ \bm{\Psi}_\tau^M \circ \bm{\Psi}_{\tau/2}^{M-1}\circ \cdots \circ \bm{\Psi}_{\tau/2}^1,
	\label{eq:sym_comp}
\end{equation}
which yields a second order method (each $ \bm{\Psi}_\tau^i$ is an exact integrator, thus self-adjoint, $\bm{\Psi}_\tau^i=[ \bm{\Psi}_\tau^i]^*$).

A simple and natural splitting of diagonal-free systems in vector fields that can be solved exactly is that of \emph{canonical shears}, advancing one coordinate at a time, while keeping the others constant. In detail, we take $M$ to be the number of variables, and set $\mathbf{F}_i$ to be the vector field equal to $f_i$ for the $i$-th variable, and $0$ for the remaining: for $i=1,\ldots,M$,
\begin{displaymath}
	\dot{\mathbf{q}} = \mathbf{F}_i (\mathbf{q}): \qquad 
	\begin{array}{l}
	\dot{q}_i = f_i (\mathbf{q}),\\
	\dot{q}_j = 0, \qquad j=1,\ldots,M, \quad j \not = i.
	\end{array}
\end{displaymath}
It is immediate to see that the variables $q_j$, $j\not= i$ are constant, while the $i$-th variable can be integrated exactly by a step of the Forward Euler method, as $f_i$ depends solely on the other variables. In other words, the exact solution is given by
\begin{displaymath}
	\mathbf{z}=\bm{\Psi}^i_\tau(\mathbf{q})= \E^{\tau \mathbf{F}_i} \mathbf{q} : \qquad 
	\begin{array}{l}
	z_i = q_i + \tau f_i (\mathbf{q}),\\
	z_j = q_j, \qquad j=1, \ldots,M, \quad j \not = i.
	\end{array}
\end{displaymath}
Since $\bm{\Psi}_\tau^i$ solves $\dot{\mathbf{q}} = \mathbf{F}_i (\mathbf{q})$ exactly and $\mathbf{F}_i$ is divergence free, it is volume preserving. Alternatively, one can proof directly the condition \R{eq:vp} by observing that the matrix $\frac{\De \bm \Psi(\mathbf{q} )}{\De \mathbf{q}}$ differs from the identity matrix only for the off-diagonal elements in row $i$, and clearly it has determinant one. Consequently, the symmetric composition \R{eq:sym_comp}  yields an explicit,  symmetric, order two, volume preserving integrator.

It must be noted that composition methods do not necessarily preserve first integrals and/or conserved quantities. A first integral $I(t)$ will be preserved if and only if each of the flows $\bm{\Psi}_\tau^k$ also preserves $I(t)$, which is typically not the case of the shears. So, the symmetric composition method \R{eq:sym_comp} is not expected to preserve the integrals and the Casimirs of the  $\mathcal{EZ}$ discretization. However, due to the symmetry of the composition \R{eq:sym_comp}, there exist theorems that guarantee that the conserved quantities are approximately well preserved over long simulation times \cite{hairer02gni, leimkuhler04shd}.

\section{Construction of the VP numerical method}
\label{sec:constructionvp}
In the context of the QG equations, we have $\mathbf{f}(\mathbf{q}) \equiv J_\mathcal{M}(\mathbf{q})$, so that $f_i$ is the $i$th component of $J_\mathcal{M}$. Since $\bm{\psi} = \Delta_N^{-1} (\mathbf{q-h})$, the $J_\mathcal{M}$ vector fields are all quadratic in the $\mathbf{q}$ variable. This means that they can be written in the form
\begin{equation}
	f_i(\mathbf{q}) =\sum_{k,l} a^i_{k,l} (q_k - h_k) q_l = \mathbf{(q-h)}^T A^i \mathbf{q},	\qquad i=1, \ldots, N^2,
	\label{eq:fi}
\end{equation}
where $A^i= (a^i_{k,l})$. We have omitted the dependence of the matrices and their elements on the vector field $\mathcal{M}=0, \mathcal{E, Z, EZ}$ for sake of simiplicity. In order to be able to apply the canonical shears methods described above, we must make sure that the vector field is diagonal-free. To do so, we study the coefficients $a^i_{k,l}$.

The coefficients $a^i_{k,l}$ can be computed explicitly and once for all at the beginning of the numerical simulations. They do not depend on the topography or the vorticity, but only on the differential operators. Let $L_x=D_x \Delta_N^{-1}$ and $L_y=D_y \Delta_N^{-1}$. Then:
\begin{eqnarray}
	\mathcal{EZ}:\qquad 	a^i_{k,l} &=& \frac13 \Big((D_x)_{i,l}[(L_y)_{l,k}+(L_y)_{i,k}] - (D_y)_{i,l}[(L_x)_{l,k}+(L_x)_{i,k}]  \nonumber \\
	&& \qquad \mbox{}+ \underbrace{[(D_x)_{i,:}*(D_y)_{l,:}- (D_y)_{i,:}*(D_x)_{l,:} ]}_{\hbox{row vector}} \underbrace{(\Delta_N^{-1})_ {:,k}}_{\hbox{col.\ vector}} \Big), \label{eq:aikl_EZ} \\
	0: \qquad 		a^i_{k,l} &=&  (D_x)_{i,l}(L_y)_{i,k} - (D_y)_{i,l}(L_x)_{i,k} \label{eq:aikl_0}\\
	\mathcal{E}: \qquad		 a^i_{k,l} &=&  (D_x)_{i,l}(L_y)_{l,k} - (D_y)_{i,l}(L_x)_{l,k}  
	\label{eq:aikl_E} \\
	\mathcal{Z}: \qquad 		a^i_{k,l} &=& \underbrace{[(D_x)_{i,:}*(D_y)_{l,:}- (D_y)_{i,:}*(D_x)_{l,:}]}_{\hbox{row vector}}\underbrace{(\Delta_N^{-1})_ {:,k}}_{\hbox{col.\ vector}}, \label{eq:aikl_Z}
\end{eqnarray}
where  we have used Matlab semicolon notation to denote rows/columns of matrices.

We use the standard five point stencil with periodic boundary conditions.  Let $E$ the $N\times N$ shift matrix (periodic). Then,
\begin{displaymath}
	D = \frac{E-E^T}{2\delta}, \qquad D_x = I\otimes D, \quad D_y = D \otimes I.
\end{displaymath}
where $\delta =2\pi/N$. Similarly,
\begin{equation}
	D_2 = \frac{E-2I + E^T}{\delta^2}, \qquad \Delta_N = D_2 \otimes I + I \otimes D_2.
	\label{eq:D2DD}
\end{equation}
For periodic boundary conditions, the matrix $D_2$ is circulant hence singular. In turn, $\Delta_N$ is block circulant, with circulant block, and singular. Further, if the matrix $D_2$ is symmetric, also the blocks in $\Delta_N$ are symmetric.

The matrix $\Delta_N$ has a rank-one subspace corresponding to constant solutions. Thus, we need an ``inverse" that projects constant functions to zero. The Moore--Penrose pseudo-inverse matrix $\Delta_N^+$ has this property. Let 
\begin{displaymath}
	\Delta_N = V \Lambda V^T
\end{displaymath} 
the an orthogonal eigenvalue decomposition of $\Delta_N$ (which exists, since $\Delta_N$ is symmetric, hence normal). One eigenvalue of $\Delta_N$, hence one diagonal element of $\Lambda$, will be equal to zero. Define the matrix $\Lambda^+$ with diagonal elements
\begin{displaymath}
	\lambda^+_{i,i} = \left\{ \begin{array}{ll}
	1/\lambda_{i,i} & \hbox{if $\lambda_{i,i} \not=0$}\\
				0&  \hbox{ otherwise.} \end{array}\right.
\end{displaymath}
The pseudoinverse  is then defined as $\Delta_N^+= V\Lambda^+ V^{T}$.

Some important observations are in place.

\begin{prop}
Consider the $J_\mathcal{M}$ vector field in \R{eq:JMflow}, whose $i$-th component coefficients are given as in \R{eq:fi}. The $A^i$ matrices, obtained using the divided differences \R{eq:D2DD} operators, are sparse and have only a few number of columns, depending on the choice of the discretized differential operators. In particular, for the $0,\mathcal{E,Z}$ case, only 4 nonzero columns, while for the $\mathcal{EZ}$ case, 8 nonzero columns.
\end{prop}
\begin{proof}
This can be proved as follows.
For a fixed $i$, it is readily observed that $(D_x)_{i,l}$ is nonzero for two values of $l$ (and independently of $k$), and so is $(D_y)_{i,l}$, yielding possibly nonzero elements for four columns for the $0,\mathcal{E}$ cases. 
Further, the row vector $[(D_x)_{i,:}*(D_y)_{l,:}- (D_y)_{i,:}*(D_x)_{l,:} ]$ can be nonzero at at most four values of $l$ (two values from the first product, further two from the second product). This implies that the $\mathcal{Z}$-case matrices also have only four columns. Finally, for the $\mathcal{EZ}$ cases, we have the superposition of the $0,\mathcal{E}$ case and the $\mathcal{Z}$ case, giving 8 columns (at most).
\end{proof}

\begin{prop}
\label{th:prop2}
$f_i(\mathbf{q})$ does not depend on $q_i$, i.e.\ $a^{i}_{i,l} = a^i_{k,i}=0$ for all $k,l=1, \ldots, N^2$. 
\end{prop}

\begin{proof}
\underline{In the $0$ case}: for $l=i$, it is obvious that $a^i_{k,i} = 0$ since $D_x,D_y$ are skew symmetric hence they have zero diagonal elements.  In the case $k=i$, we have $(L_x)_{i,i}=(L_y)_{i,i} = 0$. We show the case $(L_x)_{i,i}=0$, the other one is similar. We have $(L_x)_{i,i} = \mathbf{e}_i^T L_x \mathbf{e}_i = \mathbf{e}_i^T D_x \Delta_N^{-1} \mathbf{e}_i =  (\mathbf{e}_i^T D_x \Delta_N^{-1} \mathbf{e}_i)^T = - \mathbf{e}_i^T \Delta_N^{-1}D_x  \mathbf{e}_i = - \mathbf{e}_i^T D_x \Delta_N^{-1} \mathbf{e}_i = -(L_x)_{i,i}$. The third last passage follows from the symmetry of $\Delta_N$ and skew-symmetry of $D_x$ and the second last follows from the fact that the differential operators $D_x$ and $\Delta_N$ commute.
Hence $(L_x)_{i,i}=-(L_x)_{i,i}=0$.

\underline{In the $\mathcal{E}$ case}: for $l=i$, this is the analogous of the $0$ case. For $k=i$, we have $a^i_{i,l} = (D_x)_{i,l} (L_y)_{l,i} - (D_y)_{i,l} (L_x)_{l,i} = (D_x)_{i,l} \sum_{m} (D_y)_{l,m} (\Delta_N^{-1})_{m,i} - (D_y)_{i,l} \sum_m (D_x)_{l,m} (\Delta_N^{-1})_{m,i}$. Because of the tensor-product structure of $D_x$ and $D_y$, it is immediate to observe that $(D_x)_{i,l} (D_y)_{i,l}=0$ for any $i,l$. Hence at most one of them is nonzero at each time.

Assume $(D_x)_{i,l}\not=0$ (hence $(D_y)_{i,l}=0$). We look at  $\sum_{m} (D_y)_{l,m} (\Delta_N^{-1})_{m,i}$. Because of the structure of $D_y$, for any fixed index $l$, there are at most two indices $m_1, m_2$ for which $(D_y)_{l, m_1}$ and $(D_y)_{l, m_2}$ are nonzero. Moreover, $(D_y)_{l, m_1}= -(D_y)_{l, m_2}$, hence it is sufficient to show that $(\Delta_N^{-1})_{m_1,i}=(\Delta_N^{-1})_{m_2,i}$. This latter statement follows because of the structure of $\Delta_N^{-1}$, which  has circulant and symmetric blocks (in addition to being block-circulant and block symmetric), since such is $\Delta_N$.  The case $(D_y)_{i,l}\not=0$ (hence $(D_x)_{i,l}=0$) can be proven in a very similar manner.

\underline{In the $\mathcal{Z}$ case}: for $l=i$, the row vector in \R{eq:aikl_Z} is obviously identically zero, since, as observed above, 
$(D_x)_{i,m} (D_y)_{i,m}=0$ for any $i,m$.

For $k=i$, the proof is a verification as in the $\mathcal{E}$ case. For any $i,l$, either the row vector is identically zero (for instance if the indices $i,l$ correspond to the same block), or, fixed $i$, there exist (at most four) values of $l$ for which the row vector 
$(D_x)_{i,:}*(D_y)_{l,:}$ has a \emph{single} nonzero element, corresponding to the index $m_1$. For the same values of $i,l$, the term  $(D_y)_{i,:}*(D_x)_{l,:}$ has exactly the same \emph{single} nonzero element, but at a position $m_2$. It is true that $(\Delta_N)^{-1}_{m_1,i} = (\Delta_N)^{-1}_{m_2,i}$. The proof is tedious and technical, but, again, it follows from the special structure of $ (\Delta_N)^{-1}$, being block circulant, with circulant blocks, block symmetric, with symmetric blocks, and its connections with convolutions.

Finally, the $\mathcal{EZ}$ case is the combination of the three cases, hence it follows from those.
\end{proof}

\begin{prop}
\label{th:Aisame}
	The matrices $A^i$ in \R{eq:fi} are the identical up to a shift on the physical domain. 
\end{prop}
\begin{proof}
The coefficient of the matrices are determined only by the relative position of $q_k$, $q_l$ with respect to the variable $q_i$. This is true because the differential operators in the definition of the vector field: the interaction between $q_i$ and $q_k$ depends only on their position in the grid. 
\end{proof}

\subsection{Complexity of the VP scheme for the $\mathcal{EZ}$ discretization}
From this point, we focus on the $\mathcal{EZ}$  case (energy and entropy-preserving discretization, whose $A^i$ matrix elements are given in \R{eq:aikl_EZ}.
The matrices $A^i$ (in fact, because of Prop~\R{th:Aisame}, they can all be identified with a single matrix, that is then shifted over the domain) as well as the differential operators are computed once and for all.
Each $f_i$ evaluation requires a vector-matrix-vector multipication, where the matrices $A^i$ have only $8$ nonzero columns and $N^2$ rows. Hence each $f_i$ requires $8N^2$ multiplications and as many additions. This must be repeated for all the $N^2$ variables, and the result doubled if a symmetric composition is used. Hence the complexity of one step of the symmetric method is $32 N^4$, counting both multiplications and additions. Thus, evaluating one step of the explicit splitting method has the same order of magnitude of computing the whole vector field $\mathbf{f}$.

\subsection{Indexing strategies for the $\mathcal{EZ}$ discretization}
\label{sec:indexing}
Given a splitting of a vector field in $M$ terms, there are $M!$  possible orderings (permutations of indices)  for the symmetric composition \R{eq:sym_comp}. A standard approach would be that of choosing a \emph{sequential} ordering. Other popular approaches are based on \emph{checkerboard} orderings. Our preliminary numerical simulations indicated that these strategies were not satisfactory, hence a better strategy for the ordering of the vector fields was needed, as each ordering will have a different effect insofar error propagation is concerned. 

To have an idea of the error propagation by the splitting method, we analyze \R{eq:sym_comp} by the \emph{symmetric BCH} formula,
\begin{displaymath}
	\E^{\tau/2 \mathbf{F}_1 }\cdots \E^{\tau/2 \mathbf{F}_M}\E^{\tau/2 \mathbf{F}_M} \cdots \E^{\tau/2 \mathbf{F}_1 } = \E^{Z(\tau)},
\end{displaymath}
where 
\begin{displaymath}
	Z(\tau) = \tau\sum_{i=1}^M \mathbf{F}_i -\frac{\tau^3}{12} \sum_{\stackrel{\mbox{\scriptsize $i,j,k=1$}}{j<i,j<k}}^M [\mathbf{F}_i, [\mathbf{F}_j, \mathbf{F}_k]] - \frac{\tau^3}{24} \sum_{\stackrel{\mbox{\scriptsize $i,k=1$}}{i<k}}^M  [\mathbf{F}_i, [\mathbf{F}_i, \mathbf{F}_k]] + O(\tau^5).
\end{displaymath}
The leading error term is governed by the double commutators of vector fields, defined as $[\mathbf{F}, \mathbf{G}] = \nabla \mathbf{F}  \; \mathbf{G} - \nabla \mathbf{G}\;\mathbf{F}$, where $\nabla \mathbf{F}$ is the Jacobian of the vector field $\mathbf{F}$.
It is immediate to observe that collecting together commuting vector fields will decrease the number of error terms, it is reasonable to collect together commuting vector fields. When the vector fields do not commute, minimizing the commutator is still a reasonable strategy. Thus we estimate the norm of the commutators and collect together terms in order of increasing non-commutativity. To this goal, after computing the matrices $A^i$, for $i=1,\ldots, N^2$, for each $A^i$, we look at the column values and the row values. If, for an index $k=1, \ldots, N^2$, $A^i(:,k) =0 $ and $A^i(k,:)=0$, then $f_i$ does not depend on $q_k$. Similarly, if $A^k (:,i) =0$ and $A^k(i,:) =0$, then $f_k$ does not depend on $q_i$, hence $\mathbf{F}_i = f_i \mathbf{e}_i^T $ and $\mathbf{F}_k = f_k \mathbf{e}_k^T$ commute.

Disregarding the dependence on the topography $\mathbf{h}$, the vector fields are of the form $\mathbf{F}_i = (\mathbf{q}^T A^i \mathbf{q}) \mathbf{e}_i$ (i.e.\ they are quadratic and advance one variable at a time, $q_i$), hence $\mathbf{F}_i$ and $\mathbf{F}_j$ will nearly commute if the coefficients in $A^i$ that invoke the $j$-th variable are small. Thus, we look at
\begin{displaymath}
	c^i_j = \sum_{k=1}^{N^2}( |A^i_{j,k}|+ |A^i_{k,j}| ).
\end{displaymath}
The variable $c^i_j$ is a measure of the independence of the vector field $\mathbf{F}_i$ on the variable $q_j$. More formally, we have
\begin{displaymath}
	\nabla \mathbf{F}_i \; \mathbf{F}_j = (\mathbf{q}^T A^i \mathbf{q}) \mathbf{e}_i \mathbf{q}^T \frac{A^i+ {A^i}^T}2 \mathbf{e}_j,
\end{displaymath}
hence, assuming $\norm{\mathbf{q}}_2 \leq 1$ and $\norm{A^j}_2 \leq a $, we can bound
\begin{displaymath}
	\norm{[\mathbf{F}_i, \mathbf{F}_j]}_2 \leq  a c^i_j
\end{displaymath}
for any $i,j=1,\ldots, N$. If $A^i_{j,k}=A^i_{k,j} = 0$ for all $k$, then $\mathbf{F}_i$ is independent of $q_j$ and $\mathbf{F}_i$ and $\mathbf{F}_j$ commute. 

The vector $\mathbf{c}^i=[c^i_1,\ldots, c^i_{N^2}]^T$ is then reshaped to the grid, giving a matrix $C^i$. The matrix $C^i$ is \emph{essentially} the same for all the variables, as the matrices $A^i$ are \emph{essentially} the same for all the variables (cfr.\ Prop.~\ref{th:Aisame}), they only differ by a shift in the location. Boundary wrapping must be considered, as we the problem has periodic boundary conditions, see Figure~\ref{fig:Ci}. 

For simplicity, we illustrate the case when $N=2^p$. The arguments extend also to the generic case $N$ even (our numerical tests are all performed with $N$ even). 
The $C^i$ matrices look like perfectly symmetrical `vulcanos'. The locations at the same distance from the centre have the same commutation weight. The commutation weight decreases with the distance to the centre and it is minimal for the point on the grid that is furthest away (considering wrap-around) from it. For the standard divided difference choice \R{eq:D2DD} we get the following commutation pattern: divide the grid in four quadrants, NW, SW, NE, SE. The flow corresponding to each variable $i$ in the NW quadrant commutes with the flow of the corresponding variable in the SE quadrant. Similarly, the flow of a variable in the SW quadrant commutes with the flow of the corresponding variable in the NE quadrant (see figures~\ref{fig:comm_DD}-\ref{fig:spycomm_DD} for an illustration in the case $N=8$, i.e.\ $8\times 8$ grid, $N^2 = 64$ variables).
Commutation can be proved, but the proof (technical, in the style of Prop.~\ref{th:prop2}) is not so relevant for the discussion.
Notice also that the closest positions (hence variables) to the centre are always those with the highest non-commutativity. This explains why a plain sequential indexing of the variables, $1,2,3, \ldots, N^2$, will typically display much larger error.

\begin{figure}[t]
	\centering
	\includegraphics*[width=.8\textwidth]{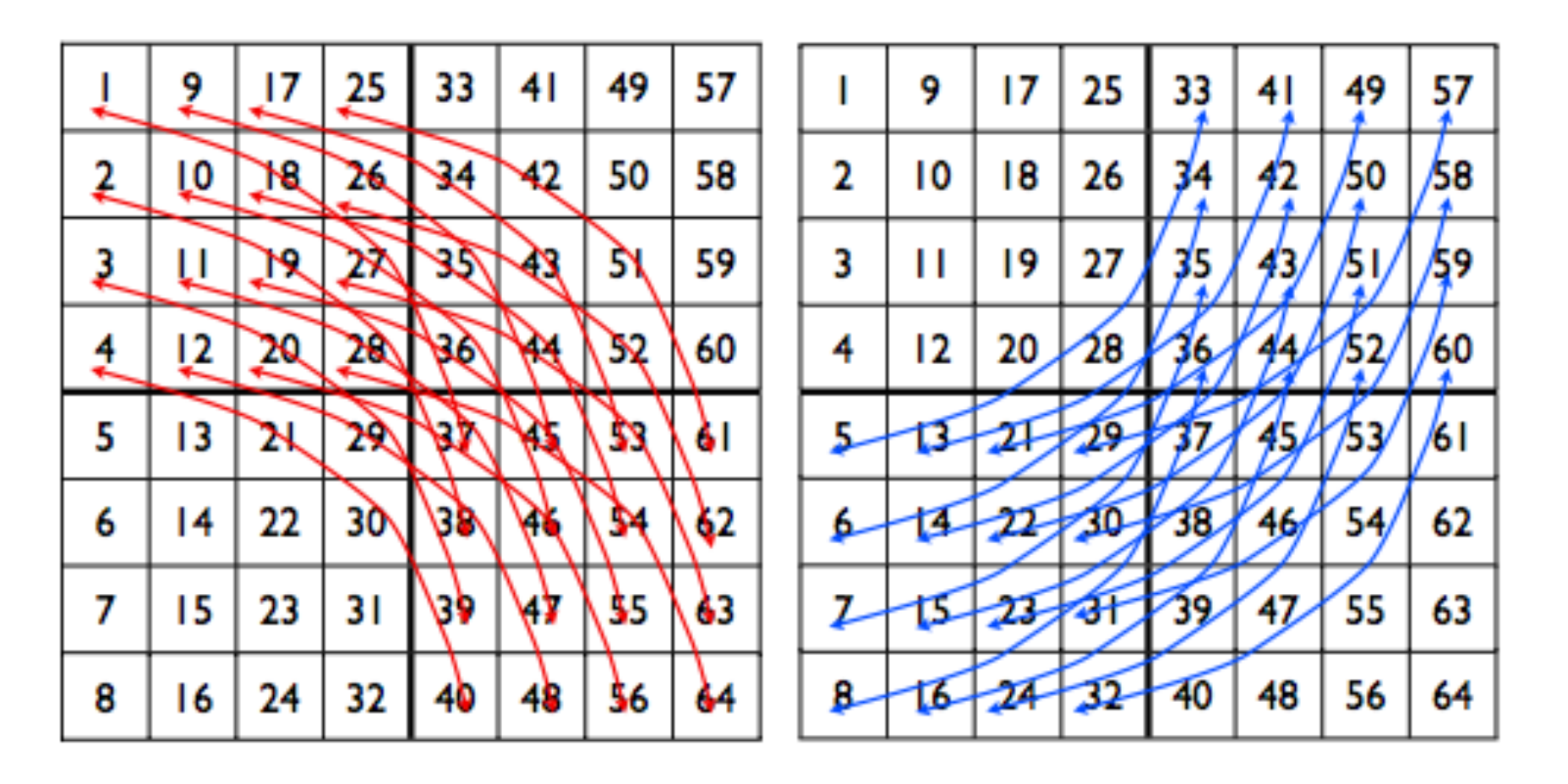}
\caption{The commutation diagram NW-SE (left) and SW-NE (right) for the vector fields $\mathbf{F}_i = f_i \mathbf{e}_i$ for the choice of the standard 5-points discretization of the Laplacian by divided differences.}
\label{fig:comm_DD}
\end{figure}

\begin{figure}[h]
	\centering
	\includegraphics*[width=.4\textwidth]{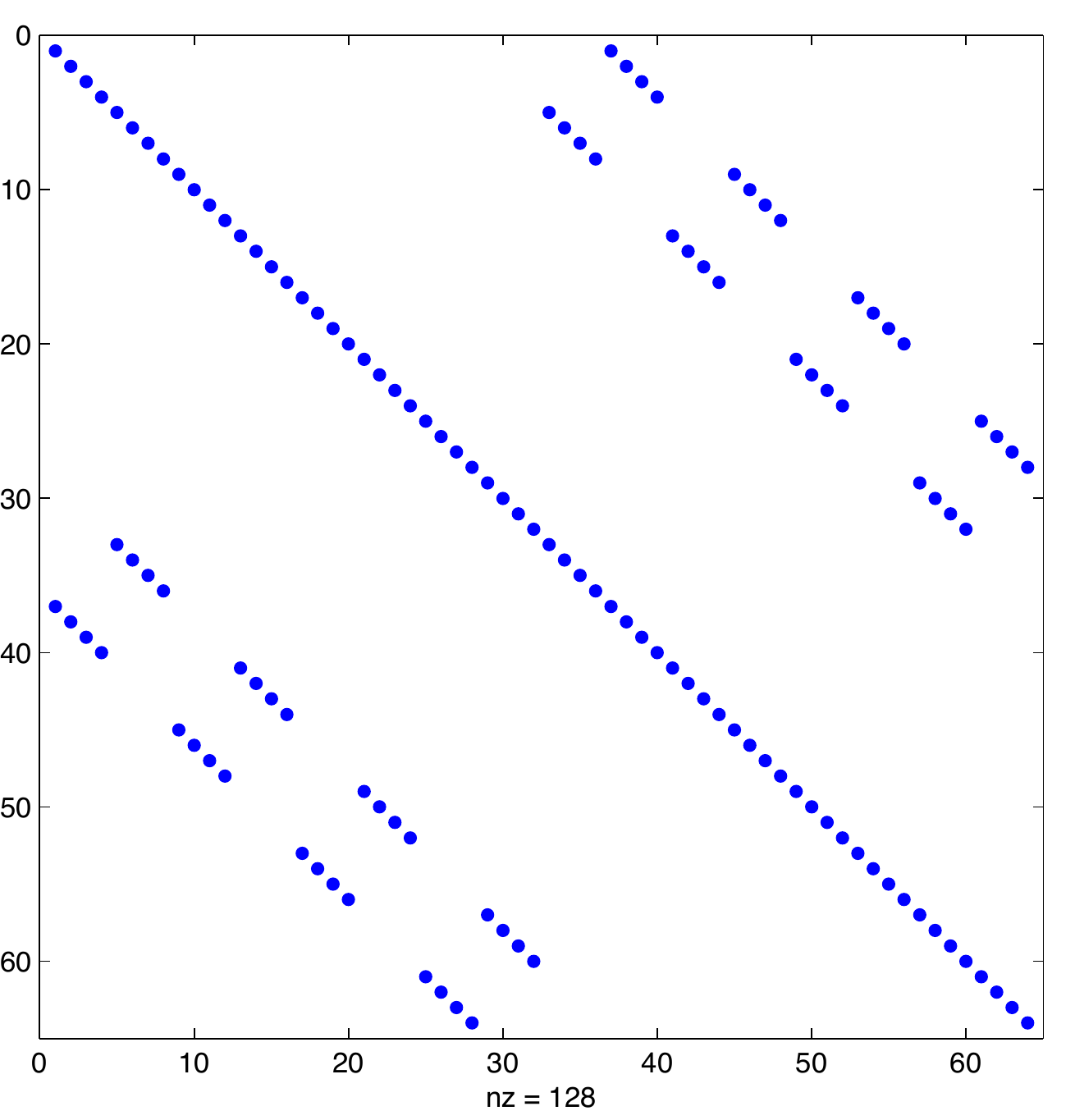} 
	\caption{A matrix showing the commuting vector fields corresponding the standard 5-points discretization of the Laplacian by divided differences. The entry in the position $(i,j)$ is 1 if $f_i$ does not depend on $q_j$, zero otherwise. This graph displays the same information as the  NW-SE and SW-NE diagrams in Figure~\ref{fig:comm_DD}. }
	\label{fig:spycomm_DD}
\end{figure}
%
\begin{figure}[ht]
\centering
\includegraphics*[width=.9\textwidth]{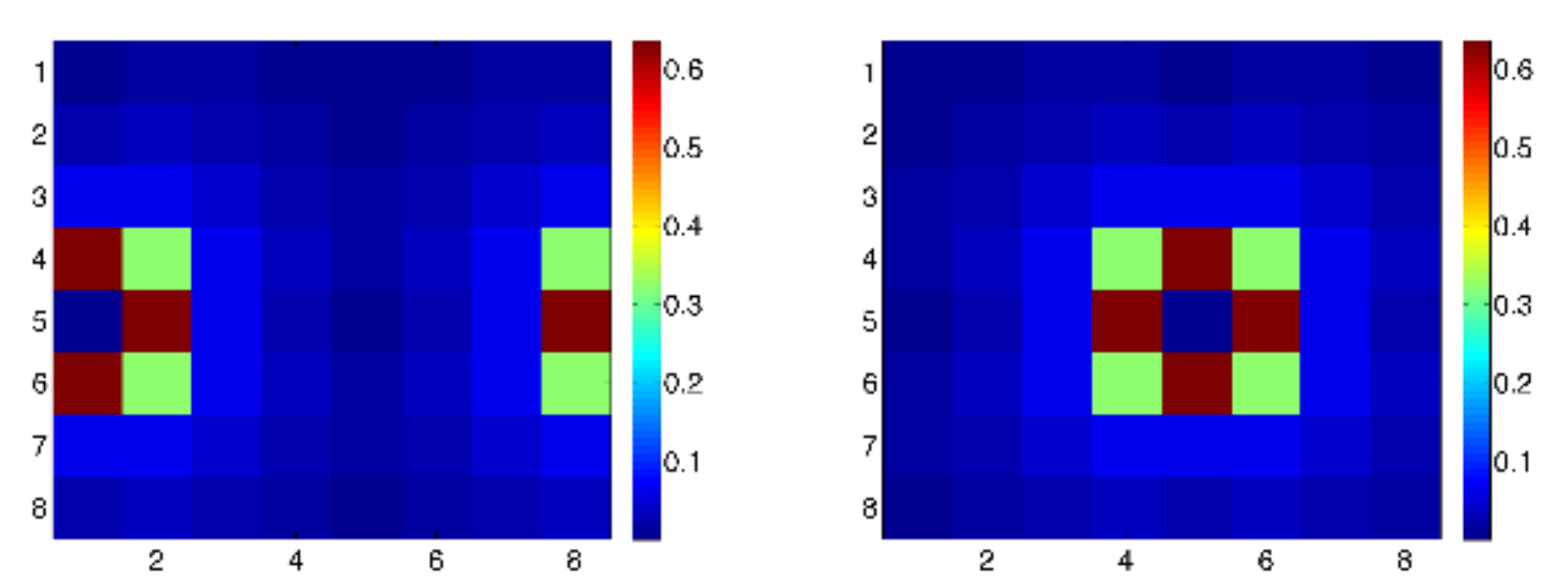}
	\caption{Illustration of the matrices $C^i$ with divided differences as in \R{eq:D2DD}, for $i=5, 37$, for a $8\times 8$ grid. The values are the same, just shifted across the domain, centered on the grid position of the $i$th variable. 
	 $F_5$ commutes with variable $F_{33}$ and $F_{37}$ commutes with $F_1$ (cfr.\ Fig.~\ref{fig:comm_DD}). Notice the wrapping along the boundaries. }
	\label{fig:Ci}
\end{figure}

Once the matrix $C^1$ is computed, each $C^i$ is obtained by shifting on the relevant position in the domain. If we consider the indexing $(1,i)$, the error will be approximately the sum of the two commutation weights, i.e.\ $C_{\mathrm{list}}=C^1+C^i$. Since we aim at minimizing these cumulative sums, we choose the variable $i_2$ to be the position where $C^1$ has a minimum value (different from the centre of the vulcano). This gives us a partial ordering $(i_1, i_2) = (1, i_2)$, with cumulative commutation weight $C_{\mathrm{list}}= C^1 + C^{i_2}$. Then we look at the values of $C_\mathrm{list}$, sorted in increasing order, ignoring the minima corresponding to $1, i_2$, and take the next minimal value. In case of a tie, we obtain a new list,  of candidate indices for which we compute the cumulative weight (starting from the first index in the new list). The new list, sorted according to the cumulative weight, is then added to the previous list, and the corresponding commutation weight matrices are added to $C_\mathrm{list}$. We look at the new minima, get a new list of candidate indices, eventually sort them, and so on, until the final ordering list is obtained. The procedure can be formalized as an algorithm.
\\
\begin{algorithm}
\label{alg:multiN}[MinCom]
Indexing of the vector fields for divided differences \R{eq:D2DD}, by minimal cumulative estimated commutation weight.

\begin{tabbing}
Description: Sorting of the vector fields. Creates a ordered list of indices \\
from $1$ to $N^2$ (number of vector fields). \\
Input: The commutation weight matrix $C^1$, $N$, starting index \texttt{i1}.\\
Output: A list of indices $(1, i_2, \ldots, i_{N^2})$ with a minimal cumulative estimated commutation weight.\\
\\
\% Initialization \\
\texttt{list =[i1]}; \\
\texttt{newlist} = \texttt{list}; \\
$C_\mathrm{list}$ = \texttt{zeros(N,N)};  \% ``cumulative commutation"  weight matrix\\
\\
\% Main loop\\
\textbf{until} \= all variable $\{1,\ldots, N^2\}$ are in \texttt{list} \textbf{do} \\
\>	\textbf{for} \= all \texttt{i} in \texttt{newlist} \\
\>\>		Shift the commutation weight matrices $C^1$ in position \texttt{i} to obtain $C^i$; \\
\>\>          	$C_\mathrm{list} = C_\mathrm{list} + C^i$;\\
\>    	\textbf{end for} \\
\>   \% Find the next most commuting indices \texttt{newlist} (candidate indices)  by getting \\
\>   \% new minimal value: \\
\>    \texttt{[sval, newlist] = sort($C_\mathrm{list}$(:))};  \% \texttt{sval} temporary variable with sorted values \\
\>    Remove from \texttt{newlist} all indices already in \texttt{list};\\
\>    Set \texttt{inewlist} to be the first index in \texttt{newlist};\\
\>    \texttt{minval = sval(inewlist)}; \% new minimum \\
\>    Find all indices in \texttt{newlist} that have value equal to \texttt{minval} and remove rest;\\
\>    Sort the element in \texttt{newlist} so that they locally \texttt{ ...} \\
\>\>	have the least cumulative commutation matrix  with starting  \texttt{ ...}\\
\>\>	index \texttt{inewlist} (recursive call) ; \\ 
\>    \texttt{list = [list, newlist]} ;  \% Concatenate lists \\
\textbf{end do}
\end{tabbing}
\end{algorithm}
The procedure for the algorithm and the intermediate cumulative weight matrices for the  $N=8$ case is illustrated in Figure~\ref{fig:alg1} below.
\begin{figure}[t]
\centering
\includegraphics*[width=.95\textwidth]{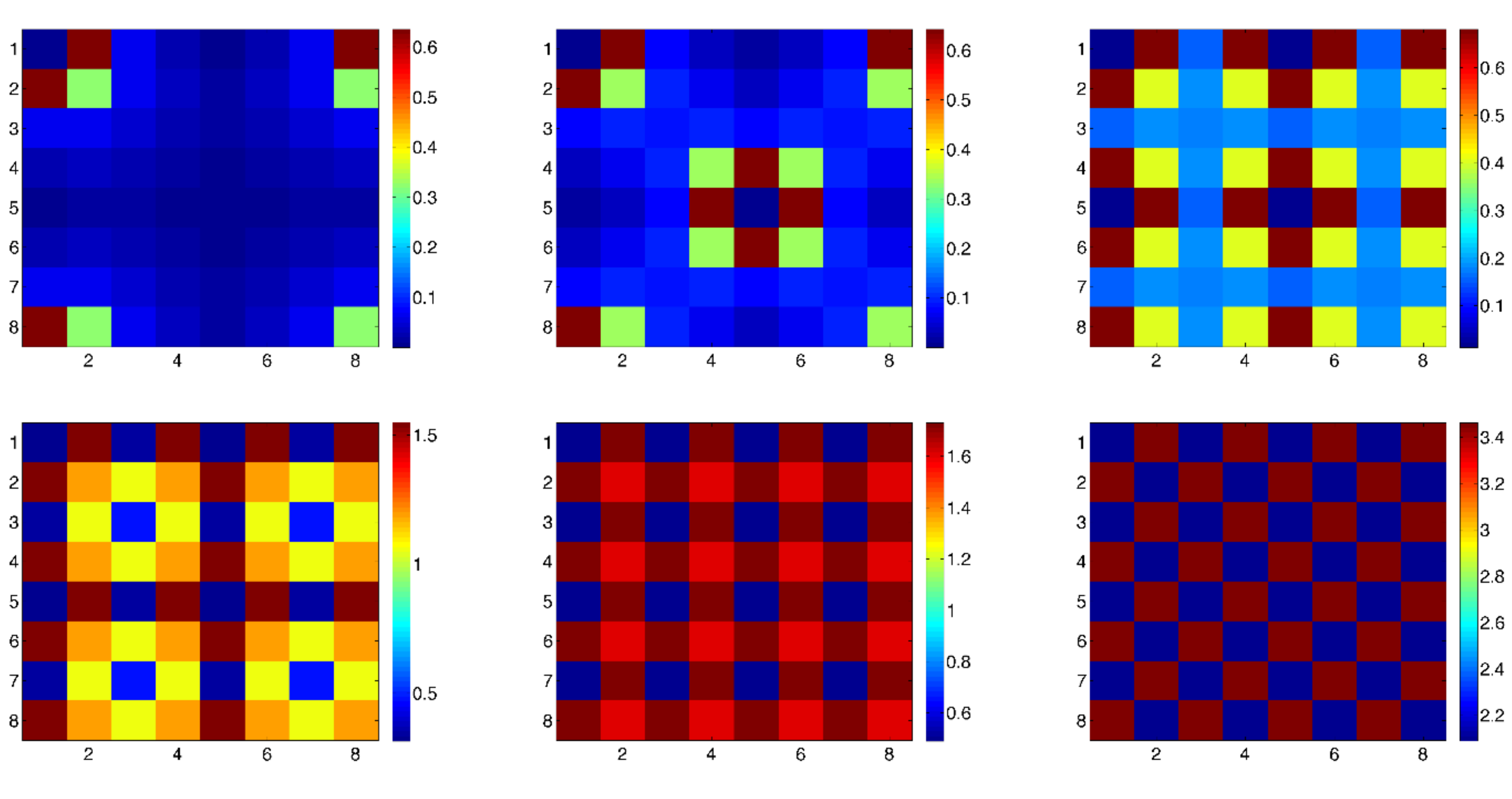}
	\caption{An illustration of Algorithm~\ref{alg:multiN} for indexing of a $8\times 8$ grid, showing the cumulative commutation matrix $C_\mathrm{list}$ with starting index $i_1=1$.  Blue values: commuting/almost commuting. Red values: high dependence on these variables (the commutator is expected to be larger). Top left to bottom right:  At the fist step, newlist=[37] is added (commuting with 1). At the second step, newlist = [5, 33]. As they too commute in between, no sorting is needed. At the third step, newlist = [3, 39,  7,  35,  21,  49,  53,  17]. They have been sorted so that the cumulative commutation matrix of the new list, starting with the $i_\mathrm{newlist}=3$, has the smallest possible value in the corresponding positions. At step 4, newlist =[19, 55,  51, 23], (also sorted). At step 5, newlist = [10,    46,    42,    14,    44,    16,    48,    12,    28,    64,    60,    32,    30,    58,    26,    62] (sorted) and finally at step 6, newlist =[2,    38,    34,     6,     8,    36,    40,     4,    56,    20,    52,    24,    22,    50,    18,    54,    29,  57,    61,    25,     9,    45,    13,    41,    47,    11,    43,    15,    27,    63,    31,    59] (sorted).}
	\label{fig:alg1}
\end{figure}
Interestingly enought, the results of this algorithm returns a grid that reminds of multigrid approaches, which might give an intuition of why this ordering should give better results. But there are differences: in the multigrid approach, the computed values are not  variable values, rather, they are transferred to the relevant variables at the higher resolution by some interpolation operator, until the finest level. In this approach, we do compute the actual variable values and no interpolation is taking place: all the computations are done at the same grid level.

\section{Numerical results}
\label{sec:numerical}
We perform numerical simulations with the $J_\mathcal{EZ}$ vector field on the test problem of \cite{abramov2003statistically,MR2442396}. The topography function is 
\begin{equation}
	h(x,y) = 0.2 \cos x + 0.4 \cos 2x.
	\label{eq:topography}
\end{equation}
The coefficients of our VP implementation based on splitting methods are given in  \R{eq:aikl_EZ}.

We test different sizes of the grid. For each grid size, we use the same initial condition, computed numerically using an optimization technique and random initial condition, with the constraints  $E= 7$, $Z =20$,  and first and third moment equal to zero as in \cite{MR2442396}. 

We estimate the mean fields of the numerical simulations and compare with the mean field predictions of the $\mathcal{EZ}$ semi-discretization. 

The numerical experiments consist in estimating the average fiels $\langle \bm{\psi} \rangle$ and the correlation $\mu$ between the average vorticity $\langle \mathbf{q} \rangle$ and $\langle \bm{\psi} \rangle$. In facts, under the assumptions of energy and enstrophy conservation, and conservation of linear momentum (total vorticity), the statistical theory asserts that
\begin{equation}
	\langle \mathbf{q}\rangle =  \mu \langle \bm{\psi} \rangle,
	\label{eq:qmupsi}
\end{equation}
and, further, that $\langle \bm{\psi} \rangle$ is a function of $x$ only (as such is $h(x,y)$). 
The least squares solution for $\mu$ gives the estimate
\begin{displaymath}
	\mu = \frac{\langle\bm{\psi} \rangle^T \langle \bm{q} \rangle}{\langle\bm{\psi} \rangle^T\langle\bm{\psi}\rangle}.
\end{displaymath}
For the values of energy and enstrophy above, $E= 7$, $Z =20$, and zero total vorticity, the predicted value of $\mu$ was estimated by \cite{MR2442396} to be $\mu \approx -0.730$ for $N=22$. We have computed the predictions for different values of $N$ below.
The  algorithm for the prediction, described in \cite{MR2442396} and adapted from the spectral analysis in \cite{carnevale87nsa}, is reported here for convenience. The ensemble averages for energy and estrophy are split into a mean part and a fluctuation, $\langle E_N \rangle = E_N(\langle \mathbf{q} \rangle ) + E'_N$, $\langle Z_N \rangle = Z_N(\langle \mathbf{q} \rangle ) + Z'_N$. Then using a spectral truncation, one has
\begin{eqnarray}
	E_N (\langle \mathbf{q}\rangle) = \frac12\sum_{k,l = -N/2+1}^{N/2} \frac{(k^2+l^2)|\hat{h}_{k,l}|^2}{(\mu+k^2+l^2)^2} \delta^2, && E'_N = \frac1{2\alpha} \sum_{k,l = -N/2+1}^{N/2}\frac{1}{\mu+k^2+l^2}, \label{eq:EM} \\
	Z_N (\langle \mathbf{q}\rangle) = \frac12\sum_{k,l = -N/2+1}^{N/2} \frac{\mu^2|\hat{h}_{k,l}|^2}{(\mu+k^2+l^2)^2} \delta^2, && Z'_N = \frac1{2\alpha} \sum_{k,l = -N/2+1}^{N/2}\frac{k^2 +l^2}{\mu+k^2+l^2}. \label{eq:ZM}
\end{eqnarray}
Starting with suitable guesses for $\mu, \alpha$, the actual values are computed by iteration (nonlinear least squares) using \R{eq:EM}--\R{eq:ZM} in conjunction with the assumption that $E^*_N = E = 7$ and $Z^*_N = Z = 20$, see Table~\ref{tab:predictedmu}.
\begin{table}
	\centering
	\begin{tabular}{l || r | r | r | r | r| r| r|} \hline
	$N$  & $6$ & 		$8$		 & $10$		& $16$		& $22$ 		& $32$	 	& $64$\\ \hline
	predicted $\mu$  & $-0.3995$& $-0.5646$ & $-0.6402$	& $-0.7106$	& $-0.7298 $ 	& $-0.7409$	& $-0.7487$ \\ \hline
	\end{tabular}
	\label{tab:predictedmu}
	\caption{Predicted values of $\mu$ as a function of $N$ for fixed $E=7$, $Z=20$ and topography $h(x,y)$ as in \R{eq:topography}.} 
		\label{tab:predictedmu}
\end{table}
%

To compute our averages, we perform a long time numerical simulation. Assuming ergodicity (because of  volume preservation of the numerical method), we compute the expectations relative to the microcanonical ensamble. On the other hand, the theory predictions determine the expectations for the canonical ensamble. Only in the limit, $N\to \infty$, the averages (hence the expectations) coincide. 

The start for averaging time is $T_0 = 10^3$, to allow decorrelation of the initial conditions. The averages are then computed at $T=10^4, 10^5, 10^6$, according to the formula
\begin{displaymath}
	\langle \psi \rangle = \frac{1}{T-T_0} \int_{T_0}^T \psi \D\, t.
\end{displaymath}
Numerically, denote by $\bm{\psi}_k$ the numerical approximation at time $t= k\tau$.  Assume $T=K\tau, T_0 = K_0\tau$. Then,
\begin{displaymath}
	\langle \bm{\psi} \rangle = \frac{1}{K-K_0} \sum_{k=K_0+1}^K \bm{\psi}_k.
\end{displaymath}
A similar formula is used for the $q$ variable.

In the experiments, we will consider the following types of  (forward) ordering of the variables:\footnote{The actual indexing is a \emph{symmetric} version of the above, i.e.\ forward and backward, in order to obtain an order 2, symmetric, splitting method.}
\begin{description}
\item[Plain] Denotes the standard sequential ordering $1,2,3,\ldots$
\item[BW] Checkerboard approach. For some types of problems, this ordering might give separation of variables and better error propagation. In this problem, the BW approach does not give variable separation, due to the ``mixing'' effect of $\Delta_N^{+}$.
\item[MinCom] The  indexing obtained by the Algorithm~\ref{alg:multiN} computing the minimal cumulative estimated commutation weight.
\end{description}

\subsection{Experiments with $8\times 8$ grid}
In the case of the `Plain' symmetric splitting, we observe a linear decrease of the relative energy over long integration time. At $T=10^6$, the relative energy error is of the order of $10^{-2}$, and this is reflected in the estimate of the $\mu$, that also increases at the same pace. The relative enstrophy error grows at a smaller rate, but still has a linear behavior. The absolute total vorticity (first moment) error also grows, at a much faster pace. 

As noted in \cite{MR2442396}, a nonzero total vorticity $\sum_{i,j} q_{i,j}$ will give a constant displacement in the relation \R{eq:qmupsi}. The total vorticity is a \emph{linear} integral, and is automatically preserved by any linear integration methods. However, this is not the case of splitting methods, unless the integral is preserved by any of the split fields.

We also plot the absolute error in the third moment. This is not as well preserved, but bounded. The result is not unreasonable as the semi-discrete equations \R{eq:JEZd}  do not preserve the third moment, while first moment, energy and enstrophy are preserved.

\begin{figure}
	\centering
\includegraphics*[width=.9\textwidth]{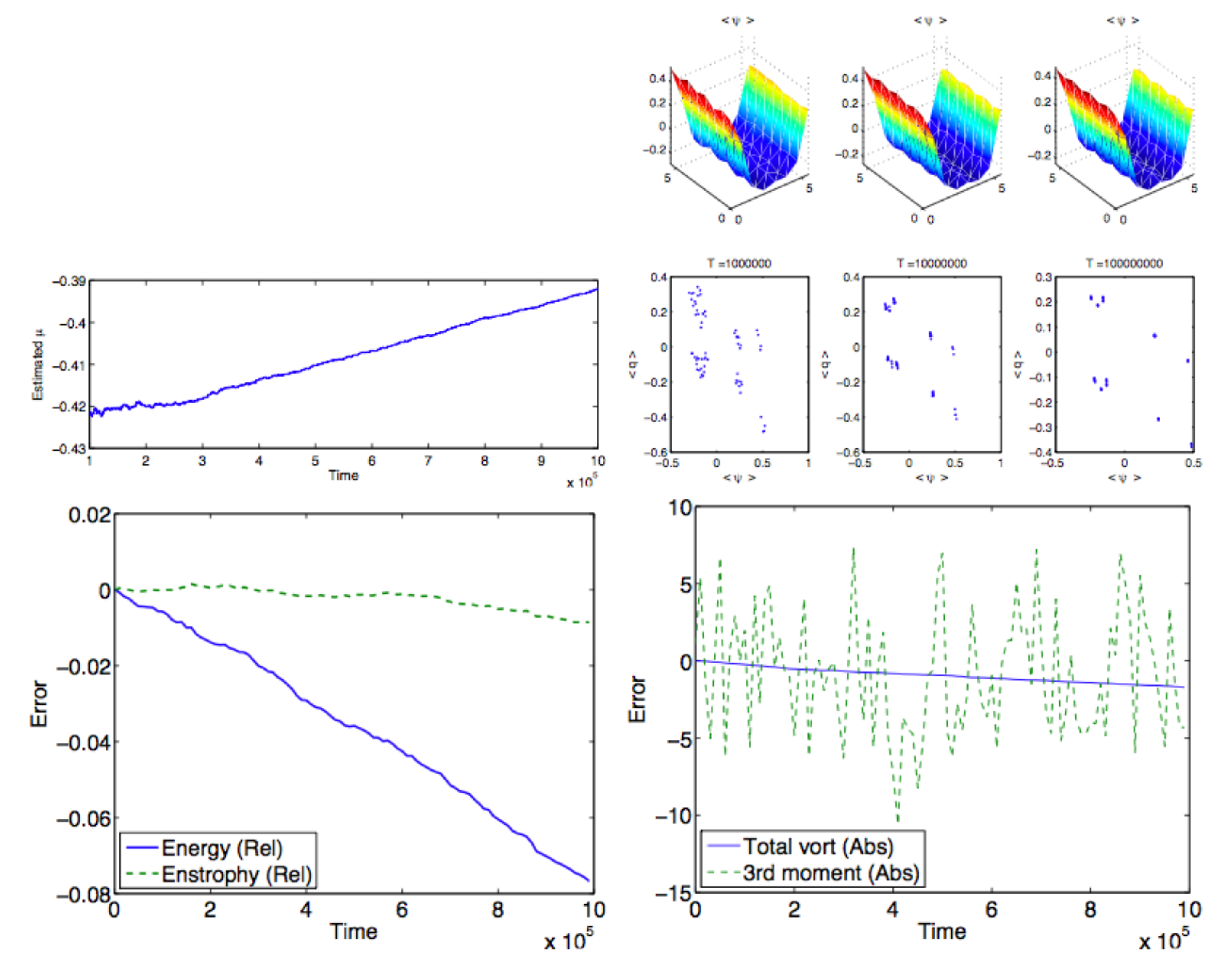}
	\caption{$8\times 8$ grid, with `Plain' ordering of the indices, $1,2,\ldots, N^2, N^2-1,\ldots, 2,1$. Top left: Estimated $\mu$. Top right: Average field, and linear fitting plots, $\langle q \rangle$ vs. $\langle \psi \rangle$. Bottom left: relative error in the energy (solid line) and enstrophy (dashed line). Bottom right: absolute errors in the total vorticity (solid line), which is also the first moment of the system, and third moment (dashed line).}
	\label{fig:8x8_Plain}
\end{figure}

\begin{figure}
	\centering
\includegraphics*[width=.9\textwidth]{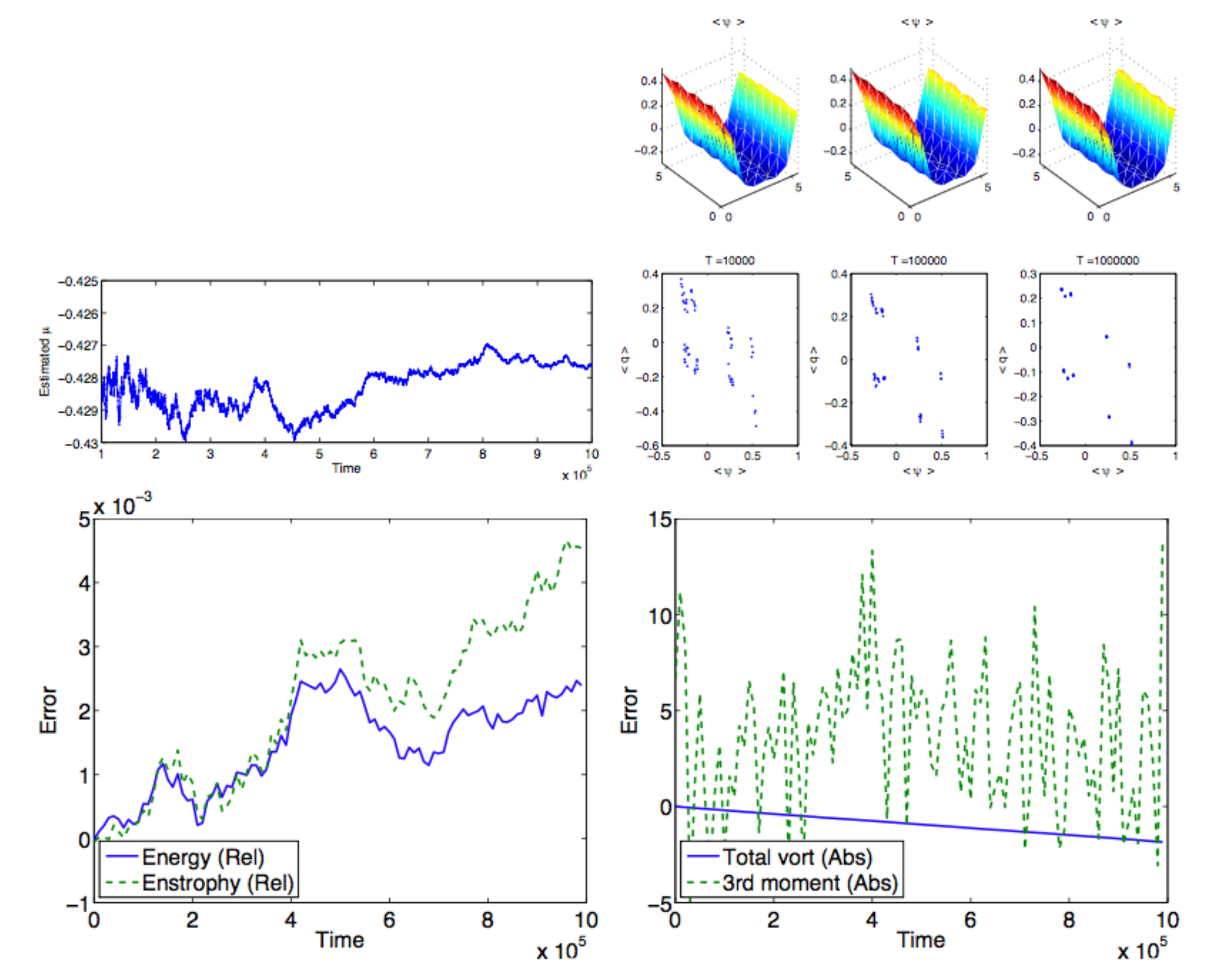}
	\caption{$8\times 8$ grid, with `BW' ordering of the indices. Top left: Estimated $\mu$. Top right: Average field, and linear fitting plots, $\langle q \rangle$ vs. $\langle \psi \rangle$. Bottom left: relative error in the energy (solid line) and enstrophy (dashed line). Bottom right: absolute errors in the total vorticity (solid line), which is also the first moment of the system, and third moment (dashed line).}
	\label{fig:8x8_BW}
\end{figure}

The checkerboard `BW' splitting is shown in Figure~\ref{fig:8x8_BW}. We observe that the energy and enstrophy are already better preserved than the `Plain' case. The relative error in the energy and enstrophy is still of the order of $10^{-3}$ even after $10^7$ integration steps. However, the total vorticity still shows a noticeable error and decreases to $\approx -2$ at the end of the integration. A comparison of the first moment for the `BW' splitting and the `MinCom' approach is shown in Figure~\ref{fig:8x8_BW_MultiN}.
The absolute error in the third moment is comparable to the case of the `Plain' discretization and appears to be bounded.

In the case of the `MinCom' symmetric splitting, we observe a much better and slower error propagation over the long time integration. The relative error in the energy and enstrophy is still of the order of $10^{-3}$ even after $10^7$ integration steps. This reflects on the  estimate of the $\mu$, that does not show a liner increase, as in the `Plain' case.
The absolute total vorticity (first moment) error also is much better conserved.
The absolute error in the third moment is comparable to the case of the 'Plain' discretization and appears to be bounded. 

\begin{figure}
	\centering
\includegraphics*[width=.90\textwidth]{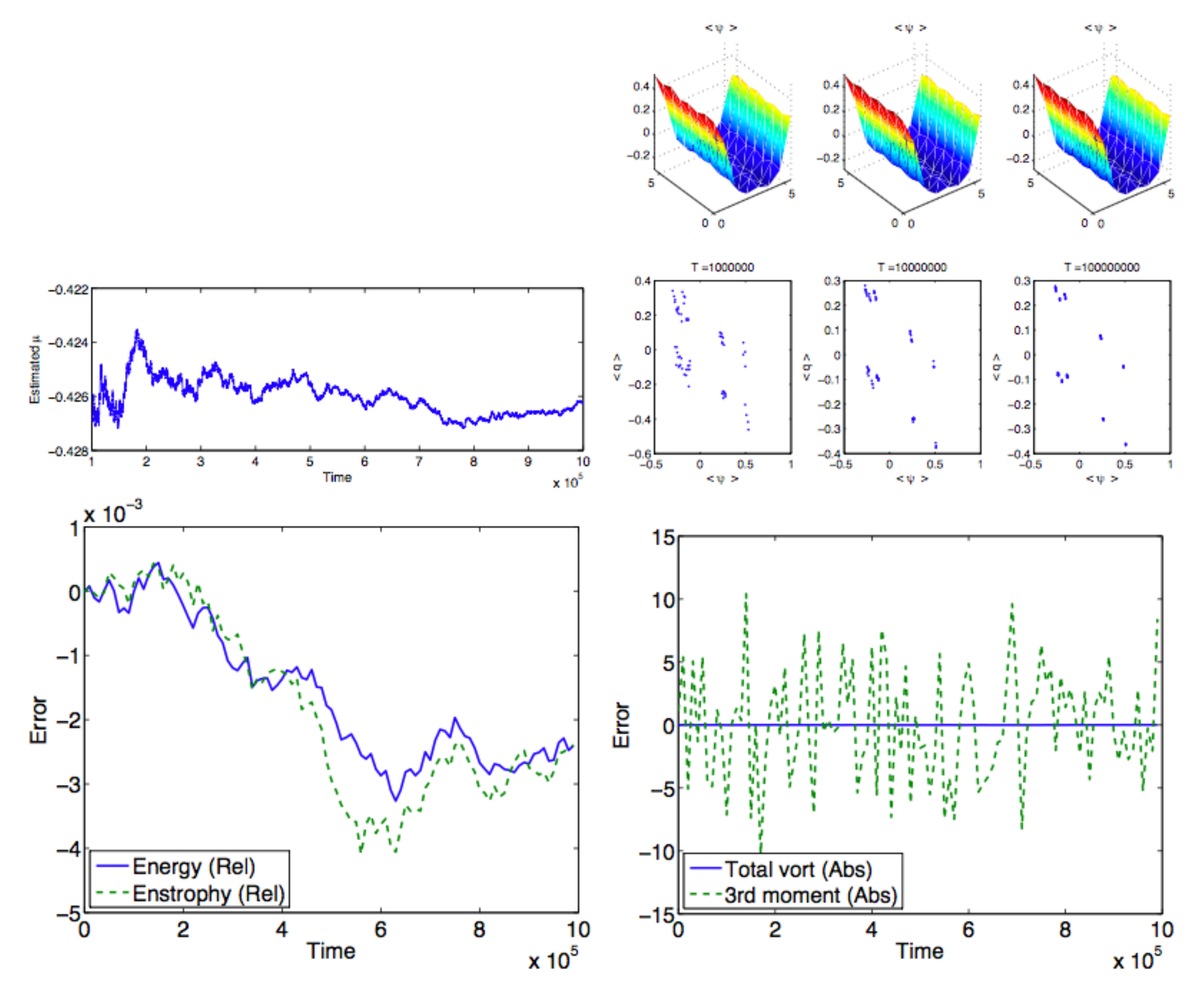}
	\caption{$8\times 8$ grid, with `MinCom' ordering of the indices, as described in Algorithm~\ref{alg:multiN}. Top left: Estimated $\mu$. Top right: Average field, and linear fitting plots, $\langle q \rangle$ vs. $\langle \psi \rangle$. Bottom left: relative error in the energy (solid line) and enstrophy (dashed line). Bottom right: absolute errors in the total vorticity (solid line), which is also the first moment of the system, and third moment (dashed line).}
	\label{fig:8x8_Plain}
\end{figure}

\begin{figure}
\includegraphics*[width=.9\textwidth]{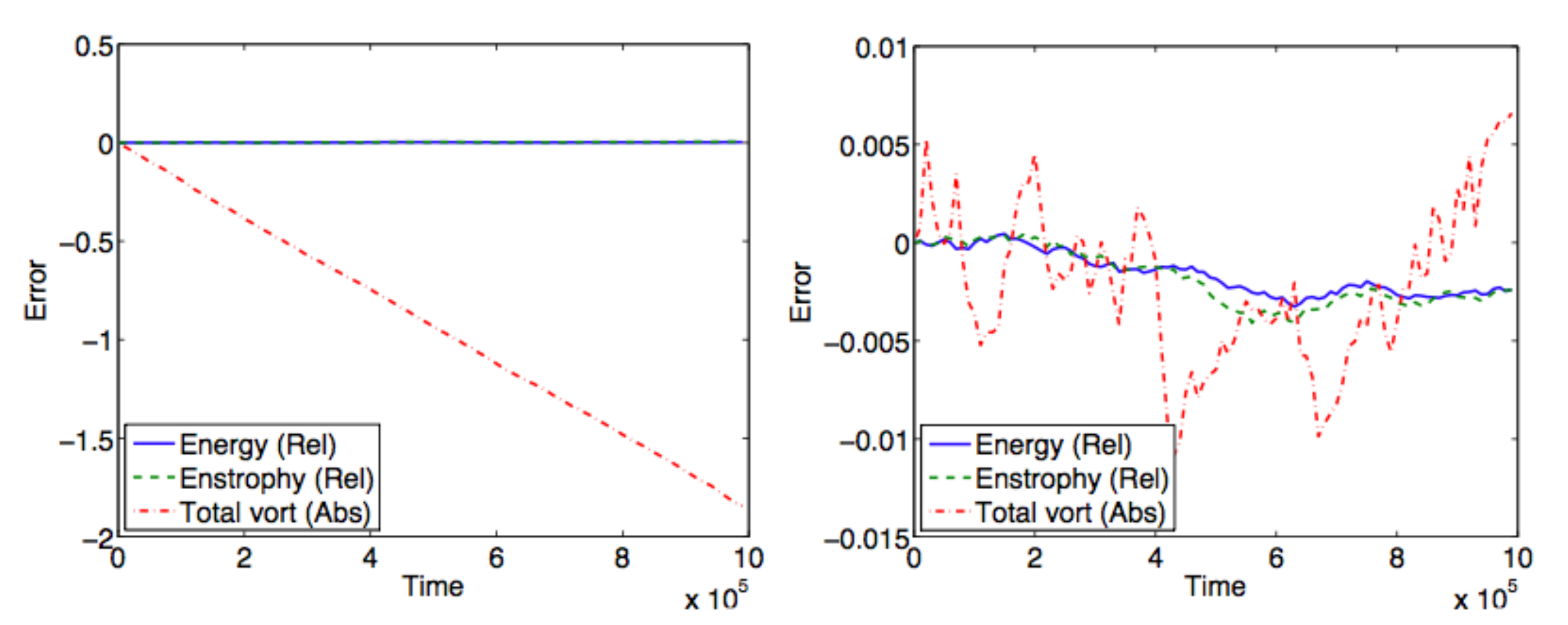}
	\caption{$8\times 8$ grid. A comparison on the conserved quantities, energy, enstrophy and total vorticity of the $\mathcal{EZ}$ semi-discretized flow \R{eq:JEZd} for the `BW' (left) and `MinCom' (right) indexing.}
	\label{fig:8x8_BW_MultiN}
\end{figure}

\begin{table}
	\centering
	\begin{tabular}{l || r | r | r  |}
	$8\times 8$&	`Plain' & `BW' & `MinCom'\\ \hline
	$T=10^4$ & $-0.4283$ &	$-0.4334$	&  $-0.4344$	 \\
	$T=10^5$ & $-0.4225$ &	$-0.4294$	& $-0.4267$	\\
	$T=10^6$ & $-0.3919$ &	$-0.4275$	& $-0.4262$	\\ \hline
	\end{tabular}
	\caption{Estimated values of $\mu$ for the different indexing, $8\times 8$ grid. Initial averaging time: $T_0= 10^3$. Stepsize of integration $\tau=0.1$.}
\end{table}

\subsection{Experiments with $16\times 16$ grid}
We repeated the same experiment on a finer grid, $N=16$. Again, we tested a `Plain' indexing, a `BW' and a `MinCom' indexing. The initial solution was the same for the three trials. Once again, the BW method failed to provide any reasonable result. The splitting using the `Plain' indexing provided results up to $t=8\times 10^5$ (blow-up time), but the solution exploded thereafter. The estimated $\mu$ shows a clear linear trend until the blow-up time. The linear trend is likely a consequence of a linear increase in energy and enstrophy by the numerical simulation.

Finally, Figure~\ref{fig:16x16_MultiN} shows the results of the simulation using the MinCom indexing approach in Algorithm~\ref{alg:multiN}. The relative errors in the energy and enstrophy, as well as linear momentum, are very well preserved, and this reflects in the estimation of $\mu$. See Table~\ref{tab:16x16} for the actual figures.

\begin{figure}
	\centering
\includegraphics*[width=.9\textwidth]{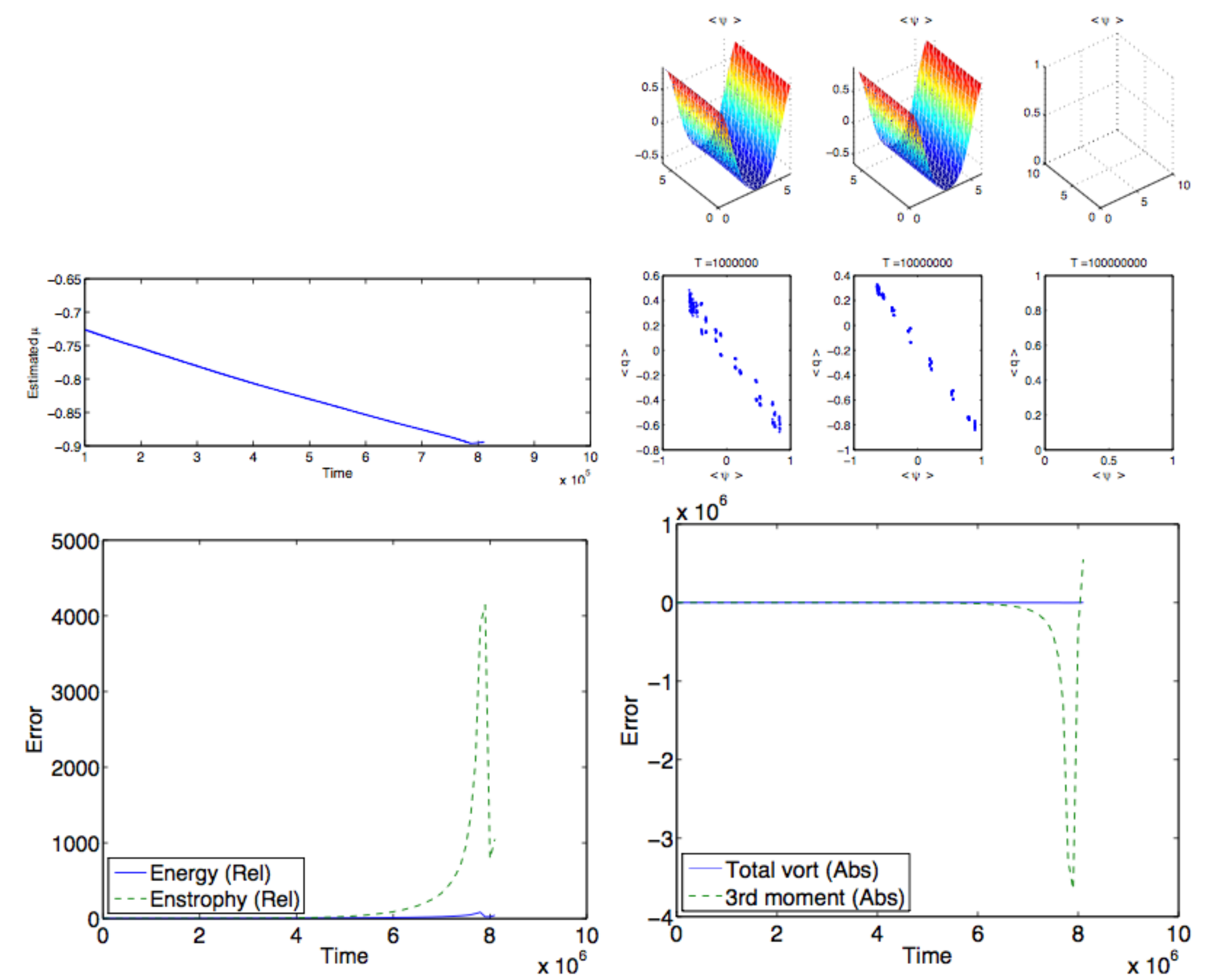}
	\caption{$16\times 16$ grid, with `Plain' ordering of the indices. Top left: Estimated $\mu$. Top right: Average field, and linear fitting plots, $\langle q \rangle$ vs. $\langle \psi \rangle$. Bottom left: relative error in the energy (solid line) and enstrophy (dashed line). Bottom right: absolute errors in the total vorticity (solid line), which is also the first moment of the system, and third moment (dashed line).}
	\label{fig:16x16_Plain}
\end{figure}

\begin{figure}
	\centering
\includegraphics*[width=.9\textwidth]{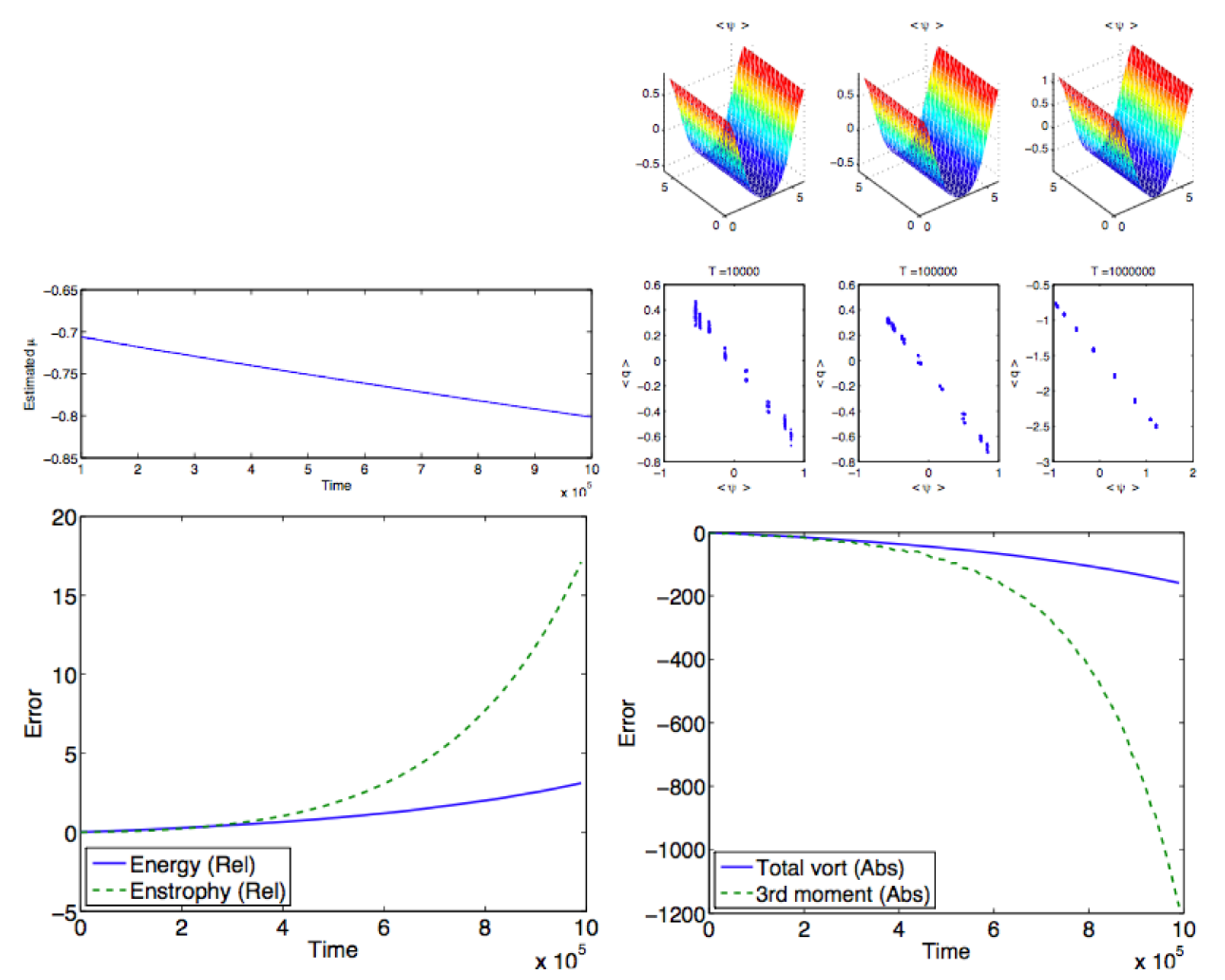}
	\caption{$16\times 16$ grid, with `BW' ordering of the indices. Top left: Estimated $\mu$. Top right: Average field, and linear fitting plots, $\langle q \rangle$ vs. $\langle \psi \rangle$. Bottom left: relative error in the energy (solid line) and enstrophy (dashed line). Bottom right: absolute errors in the total vorticity (solid line), which is also the first moment of the system, and third moment (dashed line).}
	\label{fig:16x16_Plain}
\end{figure}

\begin{figure}
	\centering
\includegraphics*[width=.9\textwidth]{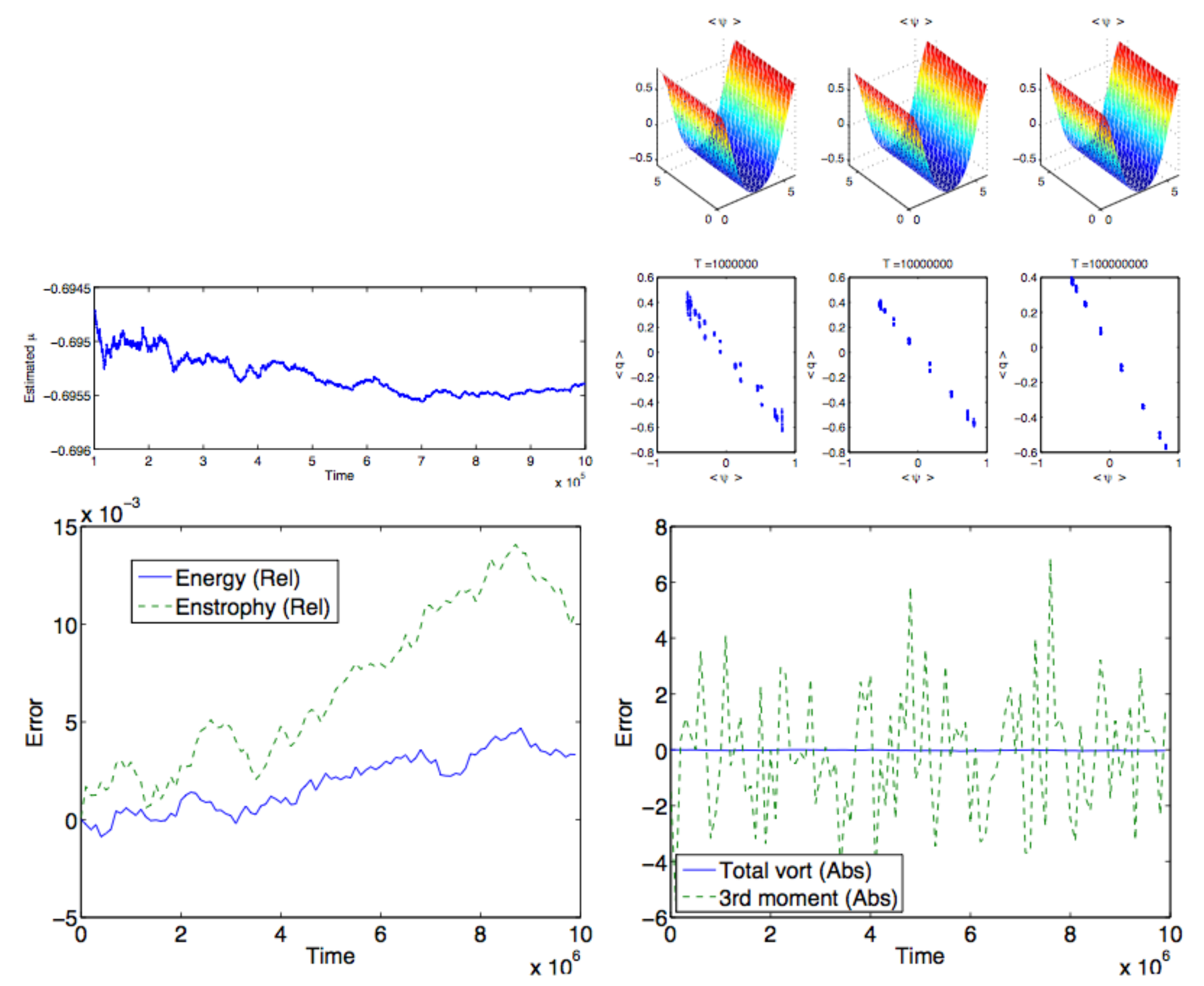}
	\caption{$16\times 16$ grid, with `MinCom' ordering of the indices, as described in Algorithm~\ref{alg:multiN}. Top left: Estimated $\mu$. Top right: Average field, and linear fitting plots, $\langle q \rangle$ vs. $\langle \psi \rangle$. Bottom left: relative error in the energy (solid line) and enstrophy (dashed line). Bottom right: absolute errors in the total vorticity (solid line), which is also the first moment of the system, and third moment (dashed line).}
	\label{fig:16x16_MultiN}
\end{figure}

\begin{table}
	\centering
	\begin{tabular}{l || r | r | r  |}
	$16\times 16$&`Plain' & `BW' & `MinCom'\\ \hline
	$T=10^4$ &  $-0.7023 $ & $-0.6951$ &$-0.6939$ 	 \\
	$T=10^5$ &  $-0.7261 $ & $-0.7061$ &$-0.6947 $	\\
	$T=10^6$ &    NaN  & $ -0.8013$ &$   -0.6954$	\\ \hline
	\end{tabular}
\caption{Estimated values of $\mu$ for the different indexing, $16\times 16$ grid. Initial averaging time: $T_0= 10^3$. Stepsize of integration $\tau=0.1$.}
	\label{tab:16x16}
\end{table}

\begin{figure}
\includegraphics*[width=.9\textwidth]{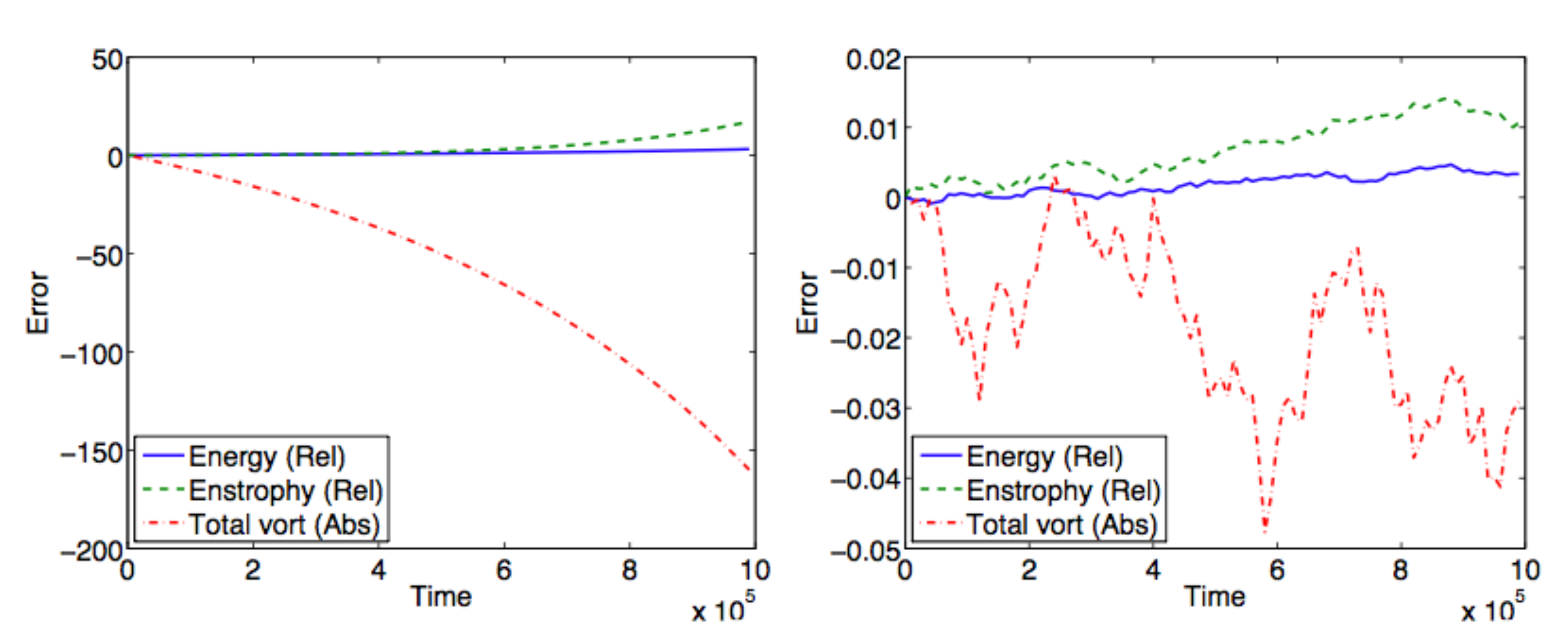}
	\caption{Energy, enstrophy and total vorticity for the `BW' (left) and `MinCom' (right) indexing for the $16\times16$ grid.}
	\label{fig:16x16_BW_MultiN}
\end{figure}

\subsection{Experiment with $22\times 22$ grid}
The same experiment is repeated on a finer grid, $N=22$. We tested only the `MinCom' indexing, given the deterioration of the quality of the solution for the two other orderings. The initial solution was randomly generated. For this number of mesh points, the energy and the enstrophy grow also for this indexing of the variables. The estimated $\mu$ shows a clear linear trend, The linear trend is likely a consequence of a linear increase in energy and enstrophy by the numerical simulation.
 See Table~\ref{tab:22x22} for the actual figures.

\begin{figure}
	\centering
\includegraphics*[width=.9\textwidth]{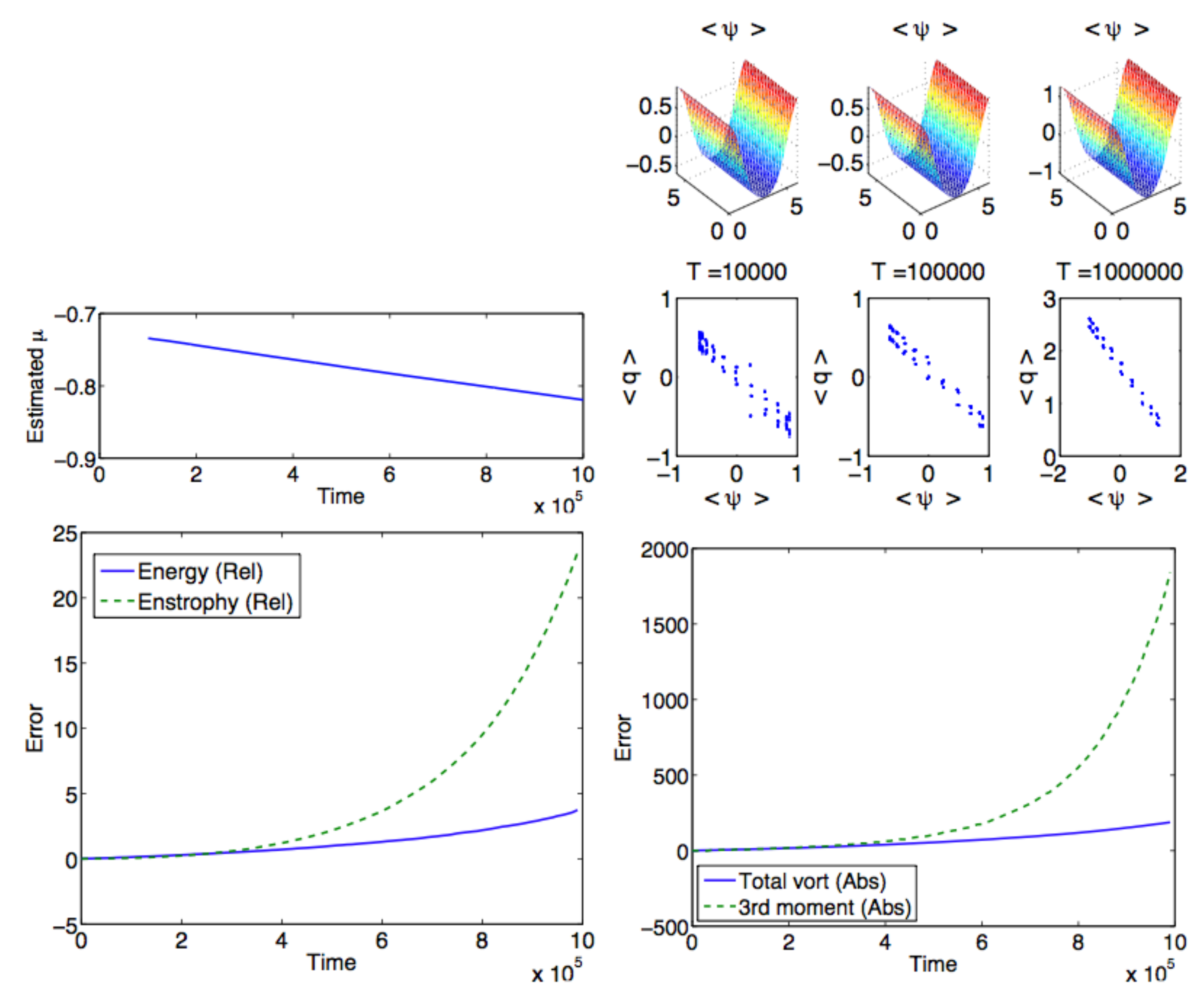}
	\caption{$22\times 22$ grid, with `MinCom' ordering of the indices, as described in Algorithm~\ref{alg:multiN}. Top left: Estimated $\mu$. Top right: Average field, and linear fitting plots, $\langle q \rangle$ vs. $\langle \psi \rangle$. Bottom left: relative error in the energy (solid line) and enstrophy (dashed line). Bottom right: absolute errors in the total vorticity (solid line), which is also the first moment of the system, and third moment (dashed line).}
	\label{fig:22x22_MultiN}
\end{figure}

\begin{table}
	\centering
	\begin{tabular}{l || r | r | r | r |}
	$22\times 22$&`Plain' & `BW' & `MinCom'      & `MinCom4'\\ \hline
	$T=10^4$ &	--	  &	--	& $-0.7248$  &	$-0.7245$    \\
	$T=10^5$ & 	--	   &	--	& $ -0.7342$ &	$ -0.7234$ \\
	$T=10^6$ & 	--	   &	--	& $ -0.8194$ & $-0.7201$	\\ \hline
	\end{tabular}
\caption{Estimated values of $\mu$ for the different indexing, $22\times 22$ grid. Initial averaging time: $T_0= 10^3$. Stepsize of integration $\tau=0.1$. The last column is obtained with an order 4 implementation of the method, using Yoshida's technique in \S\ref{sec:yoshida}.}
\label{tab:22x22}
\end{table}

\subsection{Simulations using higher order methods}
\label{sec:yoshida}
Clearly, the $22\times22$ simulation, even with the optimized ordering of the splitting terms, does not show satisfactory results, as there is a clear linear drift in the estimation of $\mu$. This seems to be apparently due to worse accumulation of error, that eventually displays an exponential growth. This is disappointing, as dimensionality seems to play a role in how long the method is close to a ``nearby problem". If the energy and the enstrophy (and linear momentum) grow too much, it is clear that the computed statistical averages will be affected. 

To enforce better preservation of the conserved quantities, we test the approach using an order four implementation of the volume preserving splitting method using Yoshida's technique \cite{yoshida90coh}: given a self-adjoint method $S^2_\tau$ of order two, it is possible to increase the order from two to four using the composition 
\begin{displaymath}
	S^4_\tau = S^2_{\alpha \tau} \circ S^2_{\beta \tau} \circ S^2_{\alpha \tau}
\end{displaymath}
where $\alpha, \beta$ obey the conditions
\begin{displaymath}
	2 \alpha + \beta = 1, \qquad 2\alpha^3 + \beta^3 = 0.
\end{displaymath}
There is a unique real solution to the above equations, given by $\alpha = 1/(2-2^{1/3})$ and $\beta = -2^{1/3}/(2-2^{1/3})$. By a similar token, it is possible to increase the order from four to eight and further, though it is known that this approach might lead to higher leading error term and compositions with fewer terms and smaller leading error can be constructed, see for instance \cite{blanes2008splitting}.

The results of the order four implementation of the volume preserving splitting method for the $22\times22$ example with the same initial condition as the order two method, are shown in Figure~\ref{fig:22x22_MultiN4}, and the computed values for $\mu$ are tabulated in Table~\ref{tab:22x22}. The fourth order  method has a better conservation of all the integrals of the $\mathcal{EZ}$ semi-discretization, and in addition it displays a good preservation of the third order momentum. However, there is still a mild linear tendency in the estimated $\mu$.

\begin{figure}
	\centering
\includegraphics*[width=.9\textwidth]{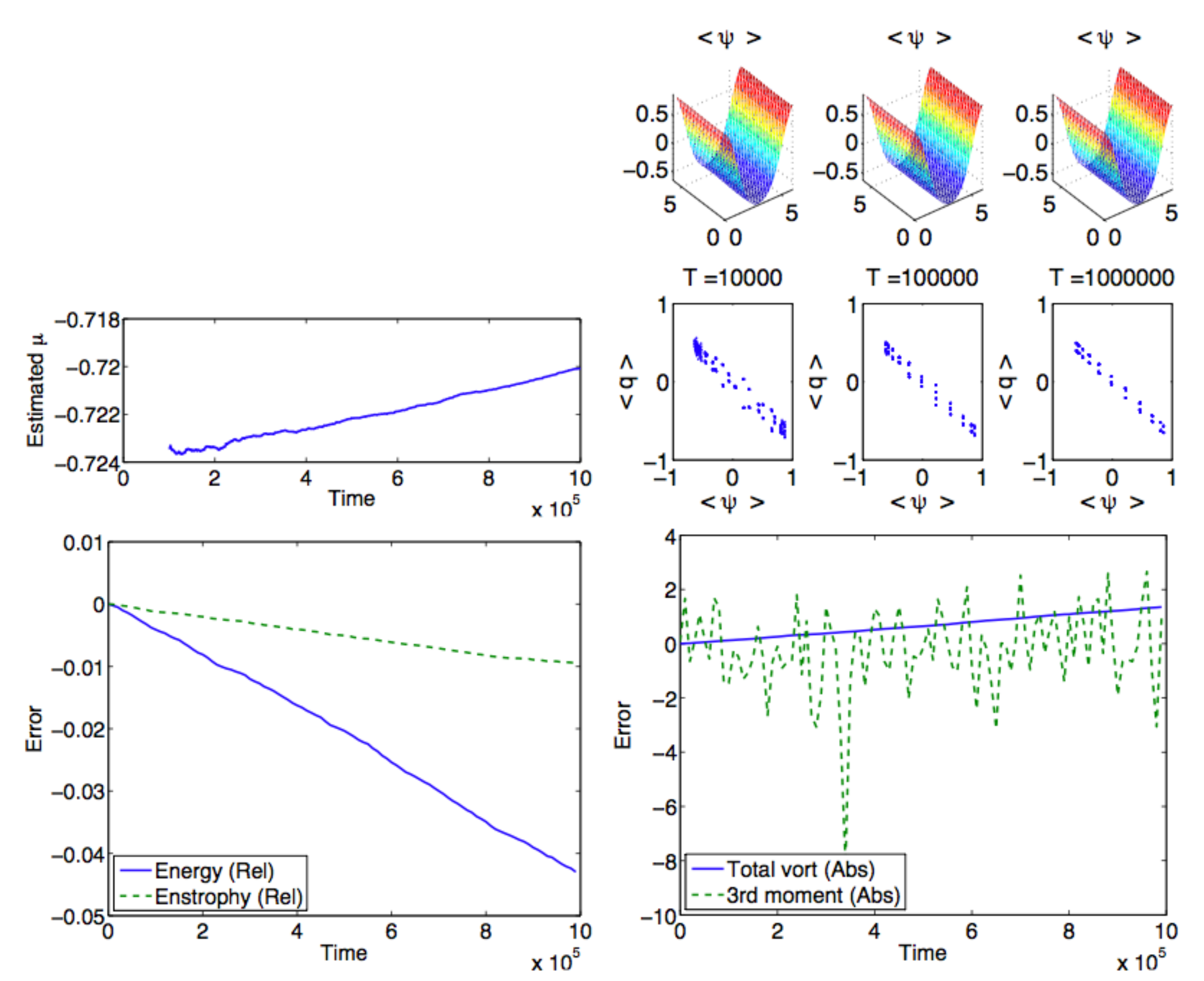}
	\caption{$22\times 22$ grid, order 4 implementation with `MinCom' ordering of the indices, using Yoshida's technique. Bottom left: relative error in the energy (solid line) and enstrophy (dashed line). Bottom right: absolute errors in the total vorticity (solid line), which is also the first moment of the system, and third moment (dashed line).}
	\label{fig:22x22_MultiN4}
\end{figure}
\begin{figure}
	\centering
	\includegraphics*[width=.5\textwidth]{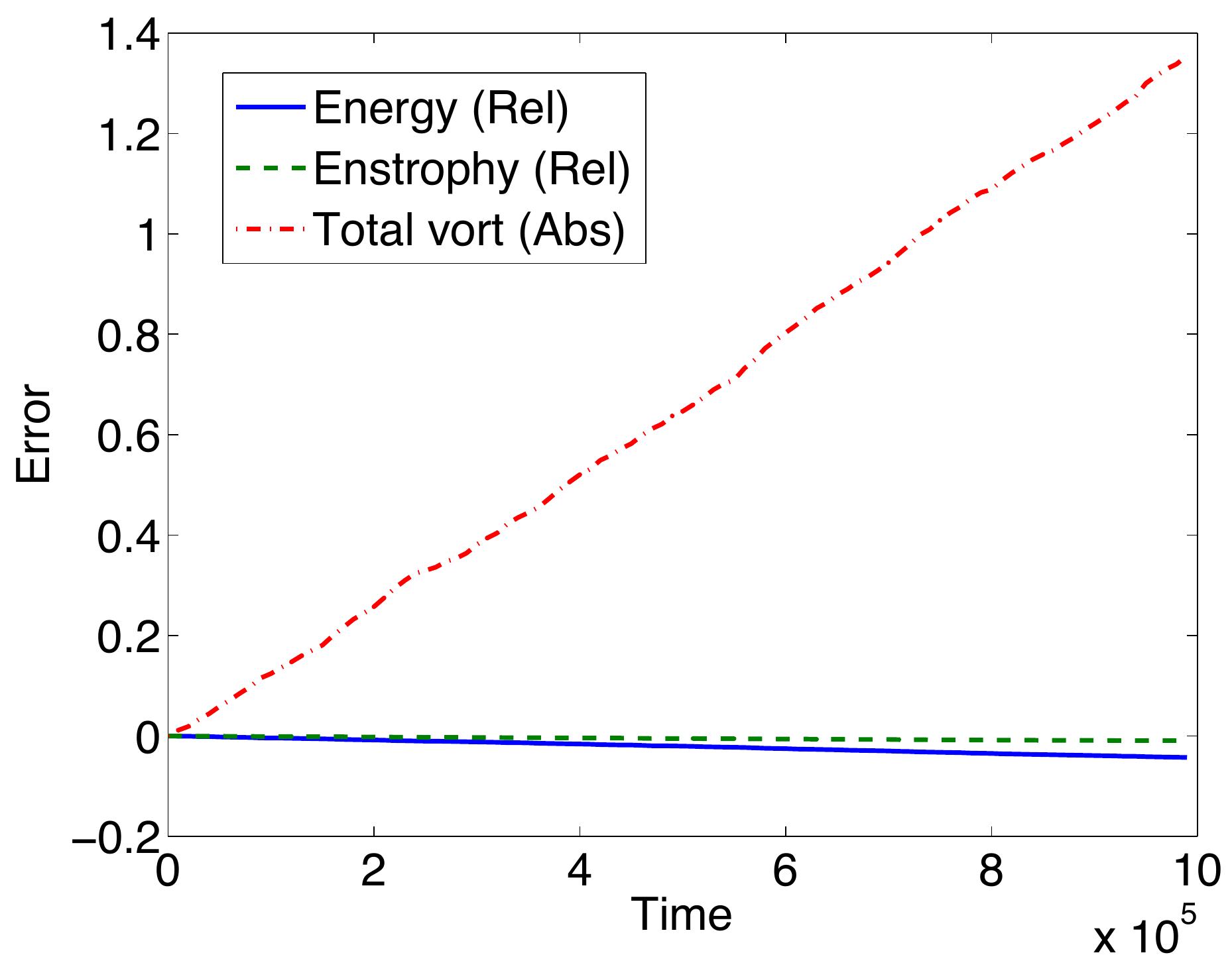}
	\caption{$22\times 22$ grid, order 4 implementation with `MinCom' ordering of the indices. Energy, enstrophy and total vorticity.}
	\label{fig:22x22_MultiN4_EZF}
\end{figure}

\subsection{Some comments and concluding remarks}
With the increase of the grid dimension, we have observed that the splitting methods display an increase in error accumulation, in particular for the total vorticity (first momentum), which is an integral of the semi-discrete problem, as are the energy and the enstrophy (in contrast to the third moment, which was not an integral of the discretization). The total vorticity had overall a consistent error propagation worse than  the energy and the enstrophy.

The statistical predictions for the QG flow are valid under the assumption that energy, enstrophy and total linear momentum are conserved. When this is not any longer the case, the numerical computed asymptotic values do not reflect any longer the theoretical predictions of the model, and the computed $\mu$ parameter shows a clear linear trend instead of converging to an asymptotic value.
Thus, particular care has to be used when using a splitting method. The ordering of the vector fields matters for error propagation and we have proposed an algorithm that produces an ordering of the vector fields that is nearly optimal. 

When a volume preserving method is used, with nearby-conservation of first integrals, the ergodicity property of the flow is respected and the departure from the nearby integrals can be used as a criterium of how good the averaged result is, in these case for which the statistical quantities depend on the first integrals of the system.
If the volume is not preserved, the preservation of integrals alone might give false positives, as generally no estimation of how much the (high-dimensional) space shrinks or expands is available, hence no indication of how reliable the estimated statistical quantities are. In \cite{MR2442396} the volume preservation of the projected Heun method and of the Implicit Midpoint Rule was studied, by computing the determinant of the Jacobian of the discrete map, $ \partial \mathbf{q}^{n+1}/\partial \mathbf{q}^n$. It was observed that a projected Heun method (preserving the integrals), computed the expected average parameters, with a far too fast convergence. Further, it was observed that the volume was strongly contracted, implying that only a smaller region of the space was visited by the trajectory. They also tested thoroughly the Implicit Midpoint Rule, which preserved intrinsically the integrals to machine precision. The IMR is a geometric integrator and backward error analysis guarantees that it solves a nearby Hamiltonian problem for exponentially long time. The IMR does not preserve volume however it is symplectic, therefore  it preserves the sum of 2-dimensional projected areas: for a shorter simulation time and a smaller number of particles ($N=12$), and it was seen to slightly increase (about $3\times 10^{-3}$ for $10^4$ steps with stepsize $0.1$).

In our experience, if the statistical parameters were required up to $d$ decimals, the relative error for the energy and enstrophy should be no larger than $10^{-d}$ in magnitude.

\section*{Acknowledgements}
We wish to acknowledge helpful discussion with Jason Frank, who first introduced us to this problem and provided us with invaluable feedback during this work, Robert McLachlan for suggesting higher order splitting method and Sebastian Reich for discussion.
\bibliographystyle{alpha}

\end{document}